\documentclass[11pt,reqno]{amsart}
\usepackage[utf8]{inputenc}
\usepackage{amsmath}
\usepackage{amsfonts}
\usepackage{amssymb}
\usepackage{color}
\usepackage{graphicx}
\usepackage{bbm}
\usepackage{comment}

\swapnumbers

\numberwithin{equation}{section}
\theoremstyle{plain}
\newtheorem{thm}{Theorem}[section]
\newtheorem{proposition}[thm]{Proposition}
\newtheorem{corollary}[thm]{Corollary}
\newtheorem{lemma}[thm]{Lemma}
\newtheorem{assumption}[thm]{Assumption}

\theoremstyle{definition}

\newtheorem{remark}[thm]{Remark}

\DeclareMathOperator{\E}{{\mathbb E}}

\DeclareMathOperator{\R}{{\mathbb R}}

\DeclareMathOperator{\N}{{\mathbb N}}

\DeclareMathOperator{\tr}{tr}

\providecommand{\eps}{\varepsilon}
\renewcommand{\phi}{\varphi}

\renewcommand{\theta}{\vartheta}

\providecommand{\norm}[1]{\lVert #1 \rVert}

\providecommand{\scapro}[2]{\langle #1,#2 \rangle}

\renewcommand{\le}{\leqslant}
\renewcommand{\ge}{\geqslant}

\allowdisplaybreaks[4]

\begin{document}

\title[Upper bounds for the reconstruction error of PCA]{Non-asymptotic upper bounds for the reconstruction error of PCA}
\author{Markus Rei\ss}
\author{Martin Wahl}
\address{Institut f\"{u}r Mathematik, Humboldt-Universit\"{a}t zu Berlin, Unter den Linden 6, 10099 Berlin, Germany.}
\email{mreiss@math.hu-berlin.de and martin.wahl@math.hu-berlin.de}
\keywords{Principal component analysis, Reconstruction error, Excess risk, Spectral projectors, Concentration inequalities}

\subjclass[2010]{Primary 62H25; secondary 15A42, 60F10}

\thanks{The research of Martin Wahl has been partially funded by Deutsche Forschungsgemeinschaft (DFG) through grant CRC 1294 ``Data Assimilation'', Project (A4) ``Nonlinear statistical inverse problems with random observations''.}

\begin{abstract}
We analyse the reconstruction error of principal component analysis (PCA) and prove non-asymptotic upper bounds for the corresponding excess risk. These bounds unify and improve existing upper bounds from the literature. In particular, they give oracle inequalities under mild eigenvalue conditions. The bounds reveal that the excess risk differs significantly from usually considered subspace distances based on canonical angles. Our approach relies on the analysis of empirical spectral projectors combined with concentration inequalities for weighted empirical covariance operators and empirical eigenvalues.
\end{abstract}

\maketitle

\section{Introduction}

Principal component analysis (PCA) and variants like functional PCA or kernel PCA are standard tools in high-dimensional statistics and unsupervised learning, see e.g. Jolliffe \cite{J02}, Horv\'{a}th and Kokoszka \cite{MR2920735} and Sch\"olkopf and Smola \cite{SS2001} for an overview. Usually, they are employed as a first step to reduce the high dimensionality of the data before methods for the specific task come into play. The basic motivation for this work is that the understanding of the error incurred by PCA in high dimensions is so far limited. In fact,  Blanchard, Bousquet, and Zwald \cite{BBZ07} exhibit upper bounds for the excess risk of the reconstruction error which give different rates in sample size and dimensionality depending on spectral properties of the covariance operator and thus exhibit complex facets of this classical statistical method. By combining spectral projector calculus with concentration inequalities, we are able to give tight bounds for the excess risk which clarify the underlying error structure. This gives rise to oracle risk bounds which in wide generality prove that the error due to projecting on empirical principal components is negligible compared to the error due to optimal dimension reduction via the population version of~PCA.

We include functional PCA and kernel PCA  in the standard multivariate PCA setting by allowing for general Hilbert spaces $\mathcal{H}$. PCA is commonly derived by minimising the reconstruction error $\E[\norm{X-PX}^2]$ over all orthogonal projections $P$ of rank $d$, where $X$ is an $\mathcal H$-valued random variable and $d$ is a given dimension. Replacing the population covariance $\Sigma$ by an empirical covariance $\hat\Sigma$, PCA computes the orthogonal projection $\hat P_{\le d}$ onto the eigenspace of the $d$ leading eigenvalues of $\hat\Sigma$. Put differently, $\hat P_{\le d}$  minimises the empirical reconstruction error and it is natural to measure its performance by the excess risk ${\mathcal E}^{PCA}_d$, that is by the difference between the reconstruction errors of $\hat P_{\le d}$ and the overall minimiser $P_{\le d}$. It is easy to see that ${\mathcal E}^{PCA}_d=\langle \Sigma, P_{\le d}-\hat{P}_{\le d}\rangle$ holds with respect to the Hilbert-Schmidt scalar product.

Comparing the excess risk ${\mathcal E}^{PCA}_d$ to the Hilbert-Schmidt distance $\norm{\hat{P}_{\le d}-P_{\le d}}_2$, which is up to a constant equal to the $l^2$-norm of the sines of the canonical angles between the corresponding subspaces, the main difference is that ${\mathcal E}^{PCA}_d$ remains small if $\hat{P}_{\le d}$ projects into eigenspaces with eigenvalues that are not much smaller than the $d$ largest ones. In the extreme case $\lambda_d=\lambda_{d+1}$, where the $d$th and $(d+1)$st largest eigenvalues coincide, the Hilbert-Schmidt distance is not even uniquely defined. Statistically, the reconstruction error is not only the basis for the very definition of PCA, but it is also more adequate for many tasks like reconstruction and prediction than the Hilbert-Schmidt distance. A typical example is given by the prediction error of principal component regression, for which Wahl \cite{W18} establishes a clear connection with the excess risk of PCA. Mathematically, an arbitrarily small spectral gap $\lambda_d-\lambda_{d+1}$ requires new techniques because spectral perturbation results deteriorate as the spectral gap shrinks. Our aim is to treat even the isotropic case  $\Sigma=\sigma^2I$, where the covariance is a multiple of the identity matrix and ${\mathcal E}^{PCA}_d=0$ holds.

Classical results for PCA provide limit theorems for the empirical eigenvalues and eigenvectors when the sample size $n$ tends to infinity, see e.g.  Anderson \cite{And} and Dauxois, Pousse and Romain \cite{DPR82}. For the Hilbert-Schmidt distance, the most well-known result is the Davis-Kahan $\sin \Theta$ theorem \cite{DK69}, which gives an upper bound in terms of the eigenvalue separation and the Hilbert-Schmidt norm of $\hat\Sigma-\Sigma$. In many cases, more precise bounds can be derived using higher-order spectral perturbation results. Nadler \cite{N08} obtains non-asymptotic bounds for the spiked covariance model and studies phase transitions when dimension and sample size tend to infinity simultaneously. Mas and Ruymgaart \cite{MR15} and Jirak \cite{M15} ask for near-optimal bounds for functional PCA with exponential or polynomial spectral decay. Koltchinskii and Lounici \cite{KL14,KL15a,KL15b,KL16} derive tight concentration bounds for the operator norm of $\hat\Sigma-\Sigma$ and study empirical spectral projectors in the so called effective rank setting.


Bounds for the reconstruction error  using the theory of empirical risk minimisation (ERM) are derived by Shawe-Taylor et al. \cite{SWCK02,SWCK05} and Blanchard, Bousquet, and Zwald \cite{BBZ07}. While \cite{SWCK05} only establishes a slow $n^{-1/2}$-rate, in \cite{BBZ07} the existence of faster rates, difficult to quantify explicitly, is discovered. We take up the ERM approach in Section \ref{SecERM} below and establish by a simple recursion argument  upper bounds, based on an interplay between a slow $n^{-1/2}$-rate and a fast $n^{-1}$-rate. These bounds clarify and partly improve the existing theory, while the proofs are short and transparent such that they have a value on their own.

Yet, we observe that the basic inequality of ERM prevents us from deriving good bounds in basic settings like isotropic covariance. In order to obtain tight bounds in more generality for ${\mathcal E}^{PCA}_d$ in Section \ref{SecNewBounds},  we employ a more sophisticated recursion argument in combination with concentration inequalities for weighted empirical covariance operators and empirical eigenvalues. This is achieved by an algebraic projector-based calculus that allows us to take advantage of the presence of the true covariance $\Sigma$ in the expression for the excess risk and to avoid difficulties arising from a straightforward application of standard perturbation theory, compare Remarks \ref{RemPT1}, \ref{RemPT2}, and \ref{RemRelativeEVCond} for more details. Considering standard examples in high-dimensional statistics and functional data analysis like spiked covariance models and exponential or polynomial eigenvalue decay, Section \ref{SecExa} shows how the general bounds apply and that existing bounds in the literature can be rediscovered and in some important aspects improved. The overall finding is that in all these cases a tight oracle inequality holds.

Finally, we discuss in Section \ref{SecDisc} how our results can be transferred to the subspace distance and, for instance, how the projector calculus yields the Davis-Kahan  $\sin\Theta$ theorem and other spectral perturbation results in a straight-forward manner. Moreover, a CLT for the excess risk is presented for fixed dimensions, acting as a benchmark for the high-dimensional results and revealing a surprising inhomogeneity of the excess risk with respect to the eigenvalue spacings. We also state the concentration inequalities for individual empirical eigenvalues that might be of independent interest. Section~\ref{SecMainTools} supplies the main tools from projector-based calculus, $\omega$-wise error decompositions and concentration inequalities.  Section \ref{SecProofs} is devoted to the proofs. In two appendices we collect proofs for the asymptotic and linear expansion results appearing in the discussion section.

\section{Main results}\label{SecMainRes}

\subsection{The reconstruction error of PCA}\label{SecRecError}

Let $X$ be a centered random variable taking values in a separable Hilbert space $(\mathcal{H},\langle\cdot,\cdot\rangle)$ of dimension $p\in\N\cup\{+\infty\}$ and let $\norm{\cdot}$ denote the norm on $\mathcal{H}$ defined by $\norm{u}=\sqrt{\langle u,u\rangle}$.

\begin{assumption}\label{SubGauss}
Suppose that $X$ is sub-Gaussian, meaning that $\mathbb{E}[\|X\|^2]$ is finite and that there is a constant $C_1$ with
\[
 \|\langle X,u\rangle\|_{\psi_2}:=\sup_{k\ge 1}k^{-1/2}\mathbb{E}\big[ |\langle X,u\rangle|^{k} \big] ^{1/k} \le C_1\mathbb{E}[\langle X,u\rangle^2]^{1/2}
\]
for all $u\in\mathcal{H}$.
\end{assumption}
If $X$ is Gaussian, then it is easy to see that Assumption~\ref{SubGauss} holds with $C_1=1$ (cf. the first formula in \cite[Equation (5.6)]{V12}).

The covariance operator of $X$ is denoted by
\begin{equation*}
\Sigma=\mathbb{E}[ X\otimes X].
\end{equation*}
By the spectral theorem there exists a sequence $\lambda_1\ge \lambda_2\ge\dots>0$ of  positive eigenvalues (which is either finite or converges to zero) together with an orthonormal system of eigenvectors $u_1,u_2,\dots$ such that $\Sigma$ has the
spectral representation
\begin{equation*}
\Sigma=\sum_{j\ge 1}\lambda_j P_j,
\end{equation*}
with rank-one projectors $P_j=u_j\otimes u_j$, where $(u\otimes v)x=\langle v,x\rangle u$, $x\in \mathcal{H}$. Note that the choice of $u_j$ and $P_j$ is non-unique in case of multiple eigenvalues~$\lambda_j$.

Without loss of generality we shall assume that the eigenvectors $u_1,u_2,\dots$ form an orthonormal basis of $\mathcal{H}$ such that $\sum_{j\ge 1}P_j=I$. We write
\[
P_{\le d}=\sum_{j\le d}P_j,\quad P_{>d}=I-P_{\le d}=\sum_{k>d}P_k
\]
for the orthogonal projections onto the linear subspace spanned by the first $d$ eigenvectors of $\Sigma$, and onto its orthogonal complement.

Let $X_1,\dotsc,X_n$ be $n$ independent copies of $X$ and let
\begin{equation*}
\hat{\Sigma}=\frac{1}{n}\sum_{i=1}^nX_i\otimes X_i
\end{equation*}
be the sample covariance. Again, there exists a sequence $\hat{\lambda}_1\ge \hat{\lambda}_2\ge\dots\ge 0$ of eigenvalues together with an orthonormal basis of eigenvectors $\hat{u}_1,\hat{u}_2,\dots$ such that we can write
\begin{equation*}
\hat{\Sigma}=\sum_{j\ge 1}\hat{\lambda}_j \hat{P}_j\text{ with } \hat{P}_j=\hat{u}_j\otimes\hat{u}_j
\end{equation*}
and
\begin{equation*}
\hat{P}_{\le d}=\sum_{j\le d}\hat{P}_j,\quad \hat{P}_{> d}=I-\hat{P}_{\le d}=\sum_{k> d}\hat{P}_k.
\end{equation*}
For linear operators $S,T:{\mathcal H}\to{\mathcal H}$ we make use of  trace and adjoint $\tr(S),S^\ast$ to define the Hilbert-Schmidt or Frobenius norm and scalar product
\[\norm{S}_2^2=\tr(S^\ast S),\quad \scapro{S}{T}=\tr(S^\ast T)\]
as well as the operator norm $\norm{S}_\infty=\max_{u\in{\mathcal H},\norm{u}=1}\norm{Su}$. For covariance operators $\Sigma$ this gives $\norm{\Sigma}_\infty=\lambda_1$ and $\norm{\Sigma}_2^2=\sum_{j\ge 1}\lambda_j^2$. Under Assumption~\ref{SubGauss}, $\Sigma$ is a trace class operator (see e.g. \cite[Theorem III.2.3]{VTC87}) and all quantities are indeed well defined. In addition, for $r\ge 1$, we use the abbreviations $\tr_{>r}(\Sigma)$ and $\tr_{\ge r}(\Sigma)$ for $\sum_{j> r}\lambda_j$ and $\sum_{j\ge r}\lambda_j$, respectively.


Introducing the class
\[{\mathcal P}_{d}=\{P:{\mathcal H}\to{\mathcal H}\,|\, P\text{ is orthogonal projection of rank }d\},\]
the (population) reconstruction error of $P\in{\mathcal P}_{d}$ is defined by
\begin{equation*}
R(P)=\mathbb{E}[ \|X-PX\|^2]=\scapro{\Sigma}{I-P}.
\end{equation*}
The fundamental idea behind PCA is that $P_{\le d}$ satisfies
\begin{equation}\label{EqVarPop}
P_{\le d}\in\operatorname{argmin}_{P\in {\mathcal P}_{d}}\limits R(P),\quad  R(P_{\le d})=\tr_{> d}(\Sigma).
\end{equation}
Similarly, the empirical reconstruction error of $P\in{\mathcal P}_{d}$ is defined by
\begin{equation*}
R_n(P)=\frac{1}{n}\sum_{i=1}^n \|X_i-PX_i\|^2=\scapro{\hat\Sigma}{I-P},
\end{equation*}
and we have
\begin{equation}\label{EqVarEmp}
\hat{P}_{\le d}\in\operatorname{argmin}_{P\in {\mathcal P}_{d}}\limits R_n(P).
\end{equation}
The excess risk of the PCA projector $\hat{P}_{\le d}$ is thus given by
\begin{equation}\label{EqExcessRisk}
{\mathcal E}^{PCA}_d:=R(\hat{P}_{\le d})-R( P_{\le d})=\langle \Sigma, P_{\le d}-\hat{P}_{\le d}\rangle.
\end{equation}
By \eqref{EqVarPop} the excess risk ${\mathcal E}^{PCA}_d$ defines a non-negative loss function in the decision-theoretic sense for the estimator $\hat P_{\le d}$ under the parameter $\Sigma$. Our main objective is to find non-asymptotic bounds for
\[
\E[{\mathcal E}^{PCA}_d]=\mathbb{E} R(\hat{P}_{\le d})-\min_{P\in\mathcal{P}_d}R(P),
\]
the decision-theoretic risk. In some situations we also consider the problem of deriving standard oracle inequalities, by allowing to replace the constant $1$ in front of the minimum by a larger constant.

Throughout the paper, $c$ and $C$ denote constants. We make the
convention that these constants are not necessarily the same at each occurrence. They usually depend on $C_1$ from Assumption \ref{SubGauss}. For our expectation bounds we make $C$ more explicit by using the constant $C_2$, where $C_2>0$ is the smallest constant such that
\begin{equation} \label{EqMomentConst}
\mathbb{E}[ \Vert P_{j}(\Sigma-\hat{\Sigma}) P_{k}\Vert_2^2]\le C_2\delta \lambda_j\lambda_k/n
\end{equation}
with $\delta=1$ if $j\neq k$ and $\delta=2$ otherwise. It is easy to check that \eqref{EqMomentConst} holds with $C_2\le 16C_1^4$ and that for $X$ Gaussian \eqref{EqMomentConst} holds with $C_2=1$.
\subsection{ERM-bounds for the excess risk}\label{SecERM}

A natural approach to derive upper bounds for the excess risk is to follow the standard theory of empirical risk minimisation (ERM). The important basic inequality in ERM is
\begin{equation}\label{EqBasicInequality}
0\le \langle \Sigma, P_{\le d}-\hat{P}_{\le d}\rangle\le \langle \Sigma-\hat{\Sigma}, P_{\le d}-\hat{P}_{\le d}\rangle=\langle \Delta, P_{\le d}-\hat{P}_{\le d}\rangle
\end{equation}
with
\[\Delta=\Sigma-\hat{\Sigma},\]
which follows from \eqref{EqVarPop} and \eqref{EqVarEmp}. This route has been taken by Blanchard, Bousquet, and Zwald \cite{BBZ07}, who applied sophisticated arguments from empirical process theory, based on  Bartlett, Bousquet, and Mendelson \cite{BBM05}. Let us derive some simple non-asymptotic expectation bounds from \eqref{EqBasicInequality}, which will set the stage for more refined results later.

\begin{proposition}\label{BBZ} We have
\begin{equation*}
{\mathcal E}^{PCA}_d\le \min\Big(\sqrt{2d}\|\Delta\|_2,\, \frac{2\|\Delta\|_2^2}{\lambda_d-\lambda_{d+1}}\Big)
\end{equation*}
with the convention $x/0:=\infty$. With Assumption \ref{SubGauss}
\begin{equation}\label{EqBBZ}
\mathbb{E}[{\mathcal E}^{PCA}_d ]\le \min\Big(\frac{\sqrt{4C_2d}\tr(\Sigma)}{\sqrt n},\,\frac{4C_2\tr^2(\Sigma)}{n(\lambda_d-\lambda_{d+1})}\Big)
\end{equation}
follows, where $C_2$ is the constant in \eqref{EqMomentConst}.
\end{proposition}

\begin{remark}\label{RemConvChoice1}
By \eqref{EqVarPop} the left-hand side in \eqref{EqBasicInequality} does not depend on the choice of $P_{\le d}$ if $\lambda_d=\lambda_{d+1}$, while the right-hand side in general does. Nevertheless, since the actual choice of $P_{\le d}$ does not alter the final result in Proposition \ref{BBZ}, we let this choice unspecified and make the convention that the $P_j$ have been fixed in advance.
\end{remark}

\begin{remark}
Extending the terminology of \cite{BBZ07}, we call the first and the second part in \eqref{EqBBZ} global and local bound, respectively, referring to the dependence on specific spectral gaps or not. The expected excess risk is thus bounded by a slow global $n^{-1/2}$-rate as well as by a fast local $n^{-1}$-rate which depends on the spectral gap $\lambda_d-\lambda_{d+1}$. For \eqref{EqBBZ} to hold only the fourth moment bound \eqref{EqMomentConst} is required instead of the full Assumption~\ref{SubGauss}.
\end{remark}

\begin{proof}[Proof of Proposition \ref{BBZ}]
From \eqref{EqBasicInequality}  we obtain
\begin{equation}\label{EqBBZProof}
({\mathcal E}^{PCA}_d)^2 \le \| \Delta\|_2^2 \|  P_{\le d}-\hat{P}_{\le d}\|_2^2
\end{equation}
by the Cauchy-Schwarz inequality. Since orthogonal projectors are idempotent and self-adjoint, we have $\scapro{ P_{\le d}}{\hat{P}_{\le d}}=\norm{P_{\le d}\hat{P}_{\le d}}_2^2\ge 0$, and thus
\[
\Vert P_{\le d}-\hat{P}_{\le d}\Vert_2^2=2(d-\langle P_{\le d},\hat{P}_{\le d}\rangle)\le 2d.
\]
Insertion into \eqref{EqBBZProof} yields the first part of the bound. The second part of the bound follows from a short recursion argument. Indeed, we have
\begin{equation*}
\Vert P_{\le d}-\hat{P}_{\le d}\Vert_2^2\le\frac{2{\mathcal E}^{PCA}_d}{\lambda_d-\lambda_{d+1}},
\end{equation*}
which is a variant of the Davis-Kahan inequality and follows by simple projector calculus, see Lemma \ref{SpectralDec} and \eqref{EqERvsHS} below. We obtain
\begin{equation*}
({\mathcal E}^{PCA}_d)^2\le \|\Delta\|_2^2\frac{2{\mathcal E}^{PCA}_d}{\lambda_d-\lambda_{d+1}}.
\end{equation*}
This yields the second part of the bound. Finally, the expectation bound \eqref{EqBBZ} follows from inserting \eqref{EqMomentConst}.
\end{proof}

The global rate can be improved by using the variational characterisation of partial traces again. In the case
$\Sigma=I+xP_{\le d}$, for instance, the global rate $p\sqrt{d/n}$ of Proposition \ref{BBZ} is improved to $d\sqrt{p/n}$. The latter is optimal for $d\le p/2$ and spectral gap $x=\sqrt{p/n}$, see the lower bound \eqref{EqLowerBound} below.

\begin{proposition}\label{GlobalBound}
Grant Assumption \ref{SubGauss}. Then we have
\begin{equation*}
\mathbb{E}[{\mathcal E}^{PCA}_d]\le C\sum_{j\le d}\max\bigg( \sqrt{\frac{\lambda_j\operatorname{tr}_{\ge j}(\Sigma) }{n}},\frac{\operatorname{tr}_{\ge j}(\Sigma)}{n}\bigg),
\end{equation*}
where $C>0$ is a constant depending only on $C_1$.
\end{proposition}

\begin{proof}
Using \eqref{EqBasicInequality}, we have
\begin{equation}\label{EqExpGlobalB}
{\mathcal E}^{PCA}_d \le \langle \Delta, P_{\le d}\rangle+\sup_{P\in\mathcal{P}_d}\langle -
\Delta, P\rangle.
\end{equation}
By the variational characterisation of partial traces (cf. \eqref{EqVarPop}, \eqref{EqVarEmp}) and the min-max characterisation of eigenvalues, see e.g. \cite[Chapter 28]{L02}, we get
\[
\sup_{P\in\mathcal{P}_d}\langle -\Delta, P\rangle\le  \sum_{j\le d}\norm{P_{\ge j}\Delta P_{\ge j}}_\infty.
\]
Noting that $\E[\langle \Delta, P_{\le d}\rangle]=0$, we conclude that
\begin{equation*}
\E[{\mathcal E}^{PCA}_d]\le \sum_{j\le d}\E\big[\norm{P_{\ge j}\Delta P_{\ge j}}_\infty\big].
\end{equation*}
Finally, we apply the moment bound for sample covariance operators obtained by Koltchinskii and Lounici \cite{KL14}. Consider  $X'=P_{\ge j}X$, $X_i'=P_{\ge j}X_i$ which again satisfy Assumption \ref{SubGauss} (with the same constant $C_1$) and lead to the covariance and the sample covariance
\begin{equation}\label{EqSigma'}
\Sigma'=P_{\ge j}\Sigma P_{\ge j},\quad \hat\Sigma'=P_{\ge j}\hat\Sigma P_{\ge j}.
 \end{equation}
Since $\Sigma'$ has trace $\operatorname{tr}_{\ge j}(\Sigma)$ and operator norm $\lambda_1'=\lambda_j$, \cite[Theorem 4]{KL14} applied to $\Delta'=\Sigma'-\hat\Sigma'$ gives
\begin{equation*}
\E\big[\norm{P_{\ge j}\Delta P_{\ge j}}_\infty\big]\le  C\max\bigg( \sqrt{\frac{\lambda_j\operatorname{tr}_{\ge j}(\Sigma) }{n}},\frac{\operatorname{tr}_{\ge j}(\Sigma)}{n}\bigg),
\end{equation*}
where $C$ is a constant depending only on $C_1$, and the claim follows.
\end{proof}

\begin{figure}[t]
\centering
\includegraphics[width=9.5cm]{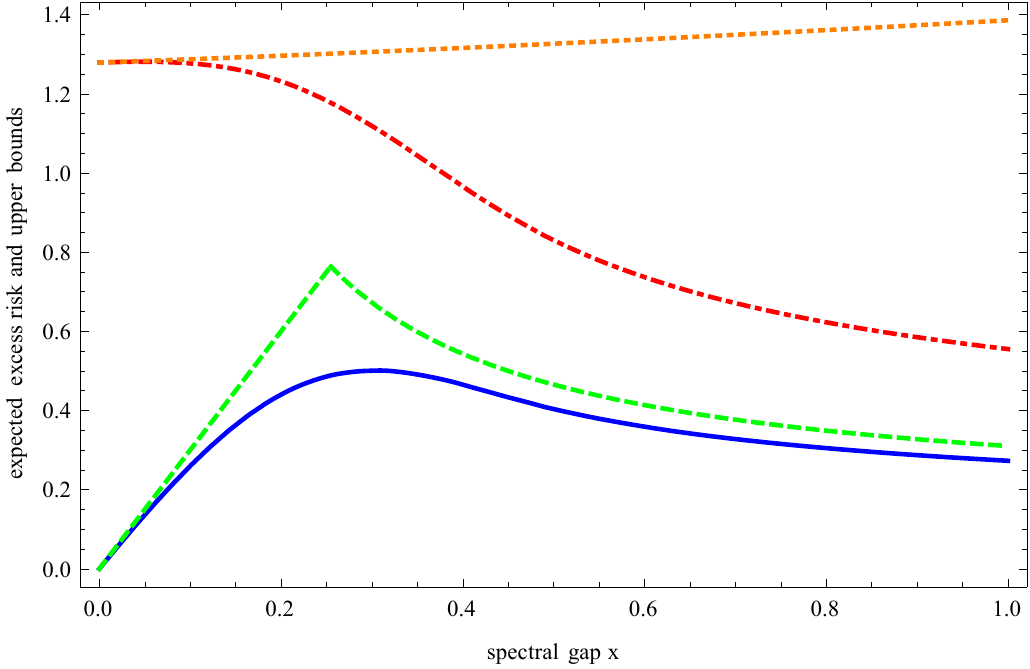}
\centering
 \caption{Expected excess risk (solid) and its upper bounds from \eqref{EqSCMAppli} (dashed), \eqref{EqBasicInequality} (dashed-dotted), and \eqref{EqExpGlobalB} (dotted) as functions of the spectral gap.}\label{Fig1}
\end{figure}

The bounds in Propositions \ref{BBZ} and \ref{GlobalBound} exhibit nicely the interplay between the global $n^{-1/2}$-rate and the local $n^{-1}$-rate. At first glance, it is surprising that the bounds derived via the basic ERM-inequality may nevertheless be suboptimal. For the simple isotropic case $\Sigma=\sigma^2I$ (enforcing a finite dimension $p$) with ${\mathcal E}^{PCA}_d=0$ they only provide an upper bound of order $d\sqrt{p/n}$. The reason is an asymmetry with the risk $\scapro{\hat\Sigma}{\hat P_{\le d}-P_{\le d}}$ with the population and empirical versions exchanged, which may be much larger than the excess risk.

For the lower bound model $\Sigma=I+xP_{\le d}$ with $n=1000$, $p=50$, and $d=3$, Figure \ref{Fig1} displays the expectation (obtained from accurate Monte Carlo simulations) of the upper bound from the basic inequality \eqref{EqBasicInequality} (dashed-dotted line) and the upper bound \eqref{EqExpGlobalB}, used for proving Proposition \ref{GlobalBound} (dotted line), compared to the expected excess risk (solid line). In addition, Figure \ref{Fig1} displays the upper bound obtained in \eqref{EqSCMAppli} with $C=1.1$, taking into account Remark \ref{RemConstants} (dashed line). This new upper  bound captures correctly the small excess risk for small spectral gaps $x$.

\subsection{New bounds for the excess risk}\label{SecNewBounds}
All results presented are proved in Section \ref{SecProofs} below.
The following representation of the excess risk is fundamental for the new bounds.
\begin{lemma}\label{SpectralDec} For any $\mu\in\R$ we have
\begin{align*}
{\mathcal E}^{PCA}_d=\sum_{j\le d}(\lambda_j-\mu)\| P_j\hat{P}_{> d}\|_2^2
+\sum_{k>d}(\mu-\lambda_{k})\Vert P_k\hat{P}_{\le d}\Vert_2^2.
\end{align*}
\end{lemma}
It turns out that the two risk parts exhibit a different behaviour and we shall bound them separately. Therefore, we introduce
\begin{equation*}
{\mathcal E}^{PCA}_{\le d}(\mu)=\sum_{j\le d}(\lambda_j-\mu)\| P_j\hat{P}_{> d}\|_2^2
,\quad
{\mathcal E}^{PCA}_{> d}(\mu)=\sum_{k>d}(\mu-\lambda_{k})\Vert P_k\hat{P}_{\le d}\Vert_2^2.
\end{equation*}
Usually, we shall choose $\mu\in[\lambda_{d+1},\lambda_d]$ such that all terms are positive, but sometimes it pays off to choose a different value. Our first main result is as follows.

\begin{proposition}\label{PropMainRes1} Grant Assumption \ref{SubGauss} and let $\mu\in [\lambda_{d+1},\lambda_d]$. Then for all $r=0,\dots,d$ we have
\begin{equation*}
\mathbb{E}[{\mathcal E}^{PCA}_{\le d}(\mu)]\le C\sum_{j\le r}(\lambda_j-\mu)\frac{\lambda_j\operatorname{tr}(\Sigma)}{n(\lambda_j-\lambda_{d+1})^2}+\sum_{j=r+1}^{d\wedge (r+p-d)}(\lambda_j-\mu)
\end{equation*}
with $C=8C_2+8C_3^2$, where $C_2$ and $C_3$ are given in \eqref{EqMomentConst} and \eqref{EqKolLou}, respectively. Moreover, if $d\le n/(16C_3^2)$, then for all $l=d+1,\dots, p+1$ we have
\begin{align*}
&\mathbb{E}[{\mathcal E}^{PCA}_{> d}(\mu)]\le C\sum_{k\ge l}(\mu-\lambda_{k})\frac{\lambda_k\operatorname{tr}(\Sigma)}{n(\lambda_d-\lambda_k)^2}+\sum_{k=(d+1)\vee (l-d)}^{l-1}(\mu-\lambda_k)+R
\end{align*}
with remainder term $R=(\mu-\lambda_p)e^{ -n/(32C_3^2)}$.  For $p=\infty$  we understand $\lambda_p=0$ and $l\in\{k\in\N\,|\,k\ge d+1\}\cup\{+\infty\}$ and for $l=p=\infty$ we understand $\sum_{k=(l-d)\vee(d+1)}^{l-1}(\mu-\lambda_k)=d\mu$.
\end{proposition}

Bounds of the same order can be derived for the $L^p$-norms of ${\mathcal E}^{PCA}_{\le d}(\mu)$ and ${\mathcal E}^{PCA}_{> d}(\mu)$ with a constant $C$ depending additionally on $p$, see e.g. Lemma \ref{lemma2mthm3} for the additional arguments needed in the case $p=2$. By simple arguments, we obtain the following corollary.

\begin{corollary}\label{MainRes1Var}  Grant Assumption \ref{SubGauss} and let $\mu\in [\lambda_{d+1},\lambda_d]$. Then we have
\[
\mathbb{E}[{\mathcal E}^{PCA}_{\le d}(\mu)]\le\sum_{j\le d} \min\bigg(C\frac{\lambda_j\operatorname{tr}(\Sigma)}{n(\lambda_j-\lambda_{d+1})},\lambda_j-\lambda_{d+1}\bigg).
\]
Moreover, if $d\le n/(16C_3^2)$, then
\[
\mathbb{E}[{\mathcal E}^{PCA}_{> d}(\mu)]\le \sum_{k>d}\min\bigg(C\frac{\lambda_k\operatorname{tr}(\Sigma)}{n(\lambda_d-\lambda_k)},\lambda_d-\lambda_k\bigg)+(\lambda_d-\lambda_p)e^{ -\frac{n}{32C_3^2}}.
\]
In both inequalities we have $C=8C_2+8C_3^2$.
\end{corollary}
Summing up the inequalities in Proposition \ref{PropMainRes1} leads to an upper bound for $\mathbb{E}[{\mathcal E}^{PCA}_{d}]$ which improves the local bound of Proposition \ref{BBZ} and gives the value $0$ in the isotropic case $\Sigma=\sigma^2 I$. Furthermore, global bounds emerge as trade-off between the two terms involved in the upper bounds. More precisely, we have:

\begin{thm}\label{ThmGlobalLocalBound} Grant Assumption \ref{SubGauss} and suppose $d\le n/(16C_3^2)$.  Then we have the local bound
\begin{equation*}
\mathbb{E}[{\mathcal E}^{PCA}_d]\le C\sum_{\substack{j\le d:\\\lambda_j>\lambda_{d+1}}}\frac{\lambda_j\operatorname{tr}(\Sigma)}{n(\lambda_j-\lambda_{d+1})}+C\sum_{\substack{k>d:\\ \lambda_k<\lambda_d}} \frac{\lambda_k\operatorname{tr}(\Sigma)}{n(\lambda_d-\lambda_k)}+(\lambda_d-\lambda_p)e^{ -\frac{n}{32C_3^2}}
\end{equation*}
and the global bound
\begin{equation}\label{EqGlBound}
\mathbb{E}[{\mathcal E}^{PCA}_d]\le \sum_{j\le d}\sqrt{\frac{C\lambda_j\operatorname{tr}(\Sigma)}{n}}+\sqrt{\frac{Cd\tr_{>d}(\Sigma) \operatorname{tr}(\Sigma)}{n}}.
\end{equation}
In both inequalities we have $C=8C_2+8C_3^2$.
\end{thm}

For our second main result we impose additional eigenvalue conditions and thus improve the first bound of Proposition \ref{PropMainRes1}. A main feature is that the full trace of $\Sigma$ can be replaced by the partial trace $\tr_{>s}(\Sigma)$, which in the case $s=d$ coincides with the oracle reconstruction error.

\begin{proposition}\label{MainRes2}  Grant Assumption \ref{SubGauss}. Then for all indices $s=1,\ldots, d$ such that
\begin{equation}\label{EqThm3Cond}
\frac{\lambda_{s}}{\lambda_{s}-\lambda_{d+1}} \sum_{j\le s}\frac{\lambda_j}{\lambda_j-\lambda_{d+1}} \le n/(16C_3^2)
\end{equation}
and all $r=0,\dots,s$, we have
\begin{equation*}
\mathbb{E}[{\mathcal E}^{PCA}_{\le d}(\lambda_{d+1})]\le C\sum_{j\le r}\frac{\lambda_j\tr_{> s}(\Sigma)}{n(\lambda_j-\lambda_{d+1})}+2\sum_{r<j\le d}(\lambda_j-\lambda_{d+1})+R
\end{equation*}
with $C=16C_2+8C_3^2$ and remainder term given by
\[
R=1024C_1\sum_{j\le r}\frac{\lambda_j\operatorname{tr}(\Sigma)}{n(\lambda_j-\lambda_{d+1})}e^{-\frac{n(\lambda_{s}-\lambda_{d+1})^2}{(4C_3\lambda_{s})^2}}.
\]
\end{proposition}

In the special case $\lambda_{d+1}=\dots=\lambda_{p}$, compare the spiked covariance model below, we have ${\mathcal E}^{PCA}_{> d}(\lambda_{d+1})=0$ and thus Proposition \ref{MainRes2} yields an upper bound for the whole excess risk. In the general case, we still have the following consequence:

\begin{thm}\label{ThmLocalGlobalImpr}
Grant Assumption \ref{SubGauss} and suppose  $ \lambda_d-\lambda_{d+1}\ge c_1(\lambda_d-\lambda_p)$ with $c_1>0$.
If \eqref{EqThm3Cond} holds with $s=d$, then we have the local bound
\begin{equation*}\label{CorThm3Res}
\mathbb{E}[ {\mathcal E}^{PCA}_d]
\le \frac{C}{c_1n}\Big(\tr_{> d}(\Sigma)+\operatorname{tr}(\Sigma)e^{-c_1^2n(\lambda_{d}-\lambda_{p})^2/(C\lambda^2_{d})}\Big)\sum_{j\le d}\frac{\lambda_j}{\lambda_j-\lambda_{d+1}}.
\end{equation*}
Moreover, if $s\le d$ is the largest number such that \eqref{EqThm3Cond} is satisfied (and $s=0$ if such a number does not exist), then we have the global bound
\[
\mathbb{E}[ {\mathcal E}^{PCA}_d]\le \frac{C}{c_1\sqrt{n}}\bigg(\sqrt{\tr_{> s}(\Sigma)}+\sqrt{\tr(\Sigma)}e^{-c_1^2n(\lambda_{s}-\lambda_{p})^2/(C\lambda^2_{s})}\bigg)\sum_{j\le d}\sqrt{\lambda_j}.
\]
In both inequalities $C$ is a constant depending only on $C_1$.
\end{thm}

Finally, observe that upper bounds for the expectation of the excess risk $\E[{\mathcal E}^{PCA}_d]\le r.h.s.$ can be equivalently formulated as exact oracle inequalities $\E[R(\hat P_{\le d})]\le \min_{P\in\mathcal{P}_d}R(P)+r.h.s.$ If we give up the constant $1$ in front of the minimum, Proposition \ref{MainRes2} also leads to a third type of bound.

\begin{thm}\label{ThmRelative}
Grant Assumption \ref{SubGauss}. Then for all indices $s=1,\ldots, d$ such that \eqref{EqThm3Cond} holds, we have
\[
\mathbb{E}[R(\hat{P}_{\le d})]\le C\tr_{> s}(\Sigma)+C\tr(\Sigma)e^{-n(\lambda_{s}-\lambda_{d+1})^2/(4C_3\lambda_{s})^2}
\]
with a constant $C>0$ depending only on $C_1$.
\end{thm}
If \eqref{EqThm3Cond} holds with $s=d$, then $\tr_{>d}(\Sigma)=\inf_{P\in\mathcal{P}_d}R(P)$ and we obtain a standard oracle inequality with an exponentially small remainder term.

\subsection{Applications}\label{SecExa}

Let us illustrate our different upper bounds for three main classes of eigenvalue behaviour: exponential decay, polynomial decay, and a simple spiked covariance model. Eigenvalue structures such as exponential or polynomial decay are typically considered in the context of functional data, see e.g. \cite{HH09,MR15,M15}, spiked covariance models are often studied in the context of high-dimensional data \cite{JL09,CMW13,VL13}.

\subsubsection*{Exponential decay}\label{ExExpDecay}
Assume for some $\alpha>0$
\begin{equation}\label{EqExpDecay}
\lambda_j=e^{-\alpha j},\ \ \ j\ge 1.
\end{equation}
Then we have $ \lambda_j-\lambda_{d+1}\ge (1-e^{-\alpha})\lambda_j$ for every $j\le d$ and \eqref{EqERRed1} below gives ${\mathcal E}^{PCA}_d\le (1-e^{-\alpha})^{-1}{\mathcal E}^{PCA}_{\le d}(\lambda_{d+1})$. Hence, Corollary \ref{MainRes1Var} implies
\begin{equation*}
\mathbb{E}[ {\mathcal E}^{PCA}_d]\le C\sum_{j\le d}\min\left(1/n,e^{-\alpha j}\right)\le C \frac{d\wedge\log (en)}{n},
\end{equation*}
where $C$ (not the same at each occurrence) is a constant depending only on $C_1$ and $\alpha$. This bound improves the local bound in Proposition \ref{BBZ} (which gives $Ce^{\alpha d}/n$) and the bounds in Theorems 3.2 and 3.4 of \cite{BBZ07}, respectively.

Next, we show that this result can be much improved by applying the local bound in Theorem \ref{ThmLocalGlobalImpr}. Indeed, the left-hand side of  \eqref{EqThm3Cond} with $s=d$ can be bounded by $d(1-e^{-\alpha})^{-2}$.
Thus, assuming that this value is smaller than $n/(16C_3^2)$, we can apply the local bound in Theorem \ref{ThmLocalGlobalImpr}. The main term  is bounded by $
C(1-e^{-\alpha})^{-3}dn^{-1} e^{-\alpha (d+1)}$ and the remainder term by $1024C_1(1-e^{-\alpha})^{-2}n^{-1}\exp(-n(1-e^{-\alpha})^2/(16C_3^2))$. We conclude that there are constants $c,C>0$ depending only on $C_1$ and $\alpha$ such that
\begin{equation}\label{EqERBoundED}
\mathbb{E}[{\mathcal E}^{PCA}_d]
\le C\frac{d e^{-\alpha d}}{n},
\end{equation}
provided that $d\le cn$. Noting for the population reconstruction error
\[
R(P_{\le d})=\sum_{k>d}e^{-\alpha k}=e^{-\alpha}(1-e^{-\alpha})^{-1}e^{-\alpha d},
\]
we see that the excess risk is smaller than the oracle risk, provided that $d\le cn$. In the Appendix \ref{SecLinExp}, we derive linear expansions for the excess risk, implying that \eqref{EqERBoundED} is indeed sharp. In fact, \eqref{EqEROptBoundED} says that for $X$ Gaussian, there are constants $c,C>0$ depending only on $\alpha$ such that $\mathbb{E}[{\mathcal E}^{PCA}_d] \ge C^{-1}d e^{-\alpha d}n^{-1}$, provided that $d\le cn$.

\subsubsection*{Polynomial decay}\label{ExaPolDecay}
Assume for some $\alpha>1$
\begin{equation}\label{EqPolDecay}
\lambda_j=j^{-\alpha},\ \ \ j\ge 1.
\end{equation}
Then the local bound in Theorem \ref{ThmGlobalLocalBound} and the inequalities
\begin{equation}\label{EqEVEPD}
\sum_{j\le d}\frac{\lambda_j}{\lambda_j-\lambda_{d+1}}\le Cd\log(ed),\quad\sum_{k>d}\frac{\lambda_k}{\lambda_d-\lambda_k}\le Cd\log(ed)
\end{equation}
from \eqref{EqEVExp1} yield that there are constants $c,C>0$ depending only on $C_1$ and $\alpha$ such that
\[
\mathbb{E}[{\mathcal E}^{PCA}_d]\le C\frac{d\log (ed)}{n}
\]
for all $d\le cn$. This already improves the results obtained in \cite[Section 5]{BBZ07}, where a rate strictly between $n^{-1/2}$ and $n^{-1}$ is derived.

 Again, for large $d$, this result can be much improved by using Theorems~\ref{ThmLocalGlobalImpr} and \ref{ThmRelative}.
Choosing $s=\lfloor d/2\rfloor$, there is a constant $c$ depending only on  $C_1$ and $\alpha$ such that Condition \eqref{EqThm3Cond} is satisfied if $d\le cn$. Thus, Theorem \ref{ThmRelative} yields
\begin{equation}\label{EqOracleBoundPolyn}
\mathbb{E}[{\mathcal E}^{PCA}]\le C\tr_{>\lfloor d/2\rfloor}(\Sigma)+Ce^{-n/C}\le Cd^{1-\alpha}+Ce^{-n/C},
\end{equation}
provided that $d\le cn$.
Noting for the population reconstruction error
\[
R(P_{\le d})=\sum_{k>d}k^{-\alpha}\ge cd^{1-\alpha},
\]
we see from \eqref{EqOracleBoundPolyn} that for $d\le cn$ the excess risk is always smaller than a constant times the oracle risk.

Similarly, Proposition \ref{MainRes2} (applied with $r=s=d$), Theorem \ref{ThmLocalGlobalImpr}, and \eqref{EqEVEPD} yield
\[
\mathbb{E}[{\mathcal E}^{PCA}_{\le d}(\lambda_{d+1})]\le  C\frac{d^{2-\alpha}\log (ed)}{n},\quad \mathbb{E}[{\mathcal E}^{PCA}_{d}]\le C\frac{d^{3-\alpha}\log (ed)}{n},
\]
provided that $d^2\log (ed)\le cn$, where $c,C>0$ are constants depending only on $C_1$ and $\alpha$. In Appendix \ref{AppAppli}, we show that the first inequality also holds without the $\log (ed)$ term, and that the second inequality can be improved to the sharp bound $Cd^{2-\alpha}n^{-1}$, yet under a more restrictive condition on $d$. This leads to the conjecture that for the excess risk the bound $Cd^{2-\alpha}n^{-1}$ holds in the larger regime $d^2\log (ed)\le cn$.

\subsubsection*{Spiked covariance model}\label{ExaTwoGroups}
Let $\Theta$ be the class of all symmetric matrices whose eigenvalues satisfy
\begin{equation}\label{EqTwoGroups}
1+\kappa x\ge \lambda_1\ge \dotsc\ge \lambda_d\ge 1+ x\text{ and }\lambda_{d+1}=\dotsc=\lambda_p=1,
\end{equation}
where $x\ge 0$ and $\kappa>1$. Then it holds
\begin{equation}\label{EqSCMAppli}
\sup_{\Sigma\in\Theta}\mathbb{E}[ {\mathcal E}^{PCA}_d]
\le \min\bigg( C\kappa\frac{(1+ \kappa x)d(p-d)}{nx},d \kappa x,(p-d)\kappa x\bigg)+\kappa xe^{ -\frac{n}{32C_3^2}},
\end{equation}
provided that $d\le cn$, where $c,C>0$ are constants depending only on $C_1$. Considering separately the cases $x\le c$ and $x> c$, we see that the excess risk is always smaller than the oracle risk $R(P_{\le d})=p-d$. To prove \eqref{EqSCMAppli}, it suffices to apply Proposition \ref{PropMainRes1}. Indeed, the claim follows from applying either Lemma \ref{SpectralDec} with $\mu=1$ and the first inequality in Proposition \ref{PropMainRes1} or Lemma \ref{SpectralDec} with $\mu=1+\kappa x$ and the second inequality in Proposition \ref{PropMainRes1} (depending on whether $(1+\kappa x)d\le p-d$ or $(1+\kappa x)d>p-d$).

In fact, since
\begin{equation*}
x\|P_{\le d}-\hat{P}_{\le d}\|_2^2\le 2 {\mathcal E}^{PCA}_d\le \kappa x\|P_{\le d}-\hat{P}_{\le d}\|_2^2,
\end{equation*}
\eqref{EqSCMAppli} is equivalent to a result by Cai, Ma, and Wu \cite[Theorem 9]{CMW13}. Moreover, their minimax lower bound \cite[Theorem 8]{CMW13} (see also Vu and Lei \cite[Theorem~A.2]{VL13}) gives
\begin{equation}\label{EqLowerBound}
\inf_{\hat{P}_{\le d}}\sup_{\Sigma\in \Theta}\mathbb{E}[ \langle\Sigma,P_{\le d}-\hat{P}_{\le d}\rangle]\ge c\min\Big( \frac{(1+x)d(p-d)}{nx},dx,(p-d)x\Big),
\end{equation}
where the infimum is taken over all estimators $\hat{P}_{\le d}$ based on $X_1,\ldots,X_n$ with values in ${\mathcal P}_{d}$ and $c>0$ is a constant.

\subsubsection*{Oracle inequality}
One interesting conclusion in the above typical situations is a nonasymptotic bound by the oracle risk, more precisely:
\begin{corollary}
In the cases \eqref{EqExpDecay}, \eqref{EqPolDecay} and \eqref{EqTwoGroups}, there are constants $c,C>0$ depending only on $C_1$, $\alpha$, and $\kappa$ such that the oracle inequality
\[
\mathbb{E}[R(\hat{P}_{\le d})]\le C\cdot R(P_{\le d}),
\]
holds for all $d\le cn$.
\end{corollary}

\subsection{Discussion}\label{SecDisc}
Let us review some connections and implications.

\subsubsection*{Subspace distance versus excess risk}\label{SecHSdistance}
Many results cover the Hilbert-Schmidt distance $\norm{\hat P_{\le d}-P_{\le d}}_2$, which has a geometric interpretation in terms of canonical angles. In this direction, the most well-known bound is the Davis-Kahan sin $\Theta$ theorem, see e.g. Yu, Wang and Samworth \cite{YWS15} for a recent statistical account. More accurate bounds are derived e.g. in Mas and Ruymgaart \cite{MR15} in a functional setting and in Vu and Lei \cite{VL13} and Cai, Ma, and Wu \cite{CMW13} in a high-dimensional sparse setting.

The squared Hilbert-Schmidt distance can be written as
\begin{equation}\label{EqHSDistanceDec}
\norm{\hat P_{\le d}-P_{\le d}}_2^2=2\sum_{j\le d}\| P_j\hat{P}_{> d}\|_2^2=2\sum_{k>d}\Vert P_k\hat{P}_{\le d}\Vert_2^2,
\end{equation}
see e.g. the proof of Lemma \ref{SpectralDec}. Compared to
\begin{align*}
{\mathcal E}^{PCA}_d=\sum_{j\le d}(\lambda_j-\lambda_{d+1})\| P_j\hat{P}_{> d}\|_2^2
+\sum_{k>d}(\lambda_{d+1}-\lambda_{k})\Vert P_k\hat{P}_{\le d}\Vert_2^2
\end{align*}
from Lemma \ref{SpectralDec}, we see that the squared Hilbert-Schmidt distance and the excess risk differ in the weighting of the projector norms. In fact, we obtain
\begin{equation}\label{EqERvsHS}
\frac{2{\mathcal E}^{PCA}_{\le d}(\lambda_{d+1})}{\lambda_1-\lambda_{d+1}}\le \norm{\hat P_{\le d}-P_{\le d}}_2^2\le \frac{2{\mathcal E}^{PCA}_{\le d}(\lambda_{d+1})}{\lambda_d-\lambda_{d+1}}\le \frac{2{\mathcal E}^{PCA}_d}{\lambda_d-\lambda_{d+1}}.
\end{equation}
This means that all excess risk bounds a fortiori imply bounds on the Hilbert-Schmidt distance up to a spectral gap factor. For instance, in our setting, \eqref{EqERvsHS} implies most versions of the Davis-Kahan $\sin \Theta$ theorem, e.g. those in \cite{YWS15}, by using the basic inequality $\scapro{\Delta}{ P_{\le d}-\hat P_{\le d}}\ge {\mathcal E}^{PCA}_d$ in \eqref{EqBasicInequality} and bounding the scalar product by a Cauchy-Schwarz or operator norm inequality. In contrast, the first inequality in \eqref{EqERvsHS} does not lead to good bounds for the excess risk when $\lambda_d-\lambda_{d+1}$ is small relative to $\lambda_1-\lambda_{d+1}$. In the extreme case $\lambda_d=\lambda_{d+1}$ the Hilbert-Schmidt distance depends on the choice of $(u_d,u_{d+1})$ and is thus not even well-defined. A more sophisticated version of \eqref{EqERvsHS} is derived in Appendix \ref{SecLinExp}.

Finally, note that the Hilbert-Schmidt distance and the excess risk have different applications. For instance, bounds for the Hilbert-Schmidt distance $\|\hat P_j- P_j\|_2$ are fundamental in the analysis of several testing algorithms, see e.g. Horv\'{a}th and Kokoszka \cite{MR2920735}. On the other hand, the excess risk is more adequate for tasks like reconstruction and prediction, see e.g. Wahl \cite{W18} for the case of the prediction error of principal component regression.

\subsubsection*{Asymptotic versus non-asymptotic} For the Hilbert-Schmidt distance it is known that for $\mathcal H=\R^p$ and $X\sim N(0,\Sigma)$ with fixed $\Sigma$
in the case $\lambda_d>\lambda_{d+1}$
\begin{equation}\label{EqAsympSinTheta}
n\|\hat{P}_{\le d}-P_{\le d}\|_2^2\xrightarrow{d}2\sum_{j\le d,k>d}\frac{\lambda_j\lambda_k}{(\lambda_j-\lambda_k)^2}g_{jk}^2
\end{equation}
holds as $n\to\infty$, where $(g_{jk})_{j\le d<k}$ is an array of independent standard Gaussian random variables, see e.g. Dauxois, Pousse and Romain \cite{DPR82} and also Koltchinskii and Lounici \cite{KL14,KL15a}. The projector calculus developed in Section \ref{SecProj} allows to obtain readily the analogue of the asymptotic result \eqref{EqAsympSinTheta} for the excess risk ${\mathcal E}^{PCA}_d$ without any spectral gap condition. More precisely, we prove in Appendix \ref{ProofPropAsympt}:

\begin{proposition}\label{PropAsympt}
Let ${\mathcal H}=\R^p$ and $X\sim N(0,\Sigma)$ with $\Sigma$ fixed. As $n\to\infty$ we have for the excess risk ${\mathcal E}^{PCA}_{d,n}={\mathcal E}^{PCA}_d$
\[
n{\mathcal E}^{PCA}_{d,n}\xrightarrow{d}\sum_{\substack{j\le d, k>d:\\ \lambda_j>\lambda_k}} \frac{\lambda_j\lambda_k}{\lambda_j-\lambda_k}g_{jk}^2,
\]
where $(g_{jk})_{j\le d<k}$ are independent standard Gaussian random variables.
\end{proposition}

We see that the excess risk converges with $n^{-1}$-rate also in the case~$\lambda_d=\lambda_{d+1}$. Note, however, that the convergence cannot be uniform in the parameter~$\Sigma$ in view of the discontinuity of the right-hand side in $(\lambda_j)$. This clearly underpins the need for non-asymptotic upper bounds for the excess risk.

In certain examples, including the spiked covariance model and exponential decay of eigenvalues, the eigenvalue expression in Proposition \ref{MainRes2} (with $r=s=d$) coincides with the one in Proposition \ref{PropAsympt}. In the general case, including polynomial decay, the eigenvalue expressions differ. In Appendix \ref{SecLinExp} we derive non-asymptotic bounds which give the asymptotic leading terms in \eqref{EqAsympSinTheta} and Proposition \ref{PropAsympt}, by using linear expansions for $\hat P_{\le d}$ and $\hat{P}_{> d}$. These bounds, however, require stronger eigenvalue conditions (including $\lambda_d>\lambda_{d+1}$). In contrast, our main results in Section \ref{SecNewBounds} also apply to the case of small or vanishing spectral gaps.

\subsubsection*{Eigenvalue concentration} We obtain deviation inequalities for empirical eigenvalues which are of independent interest. Concentration inequalities for eigenvalues using tools from measure concentration are widespread, see e.g. \cite{L01,M04,L07,AGZ10,BLM13,BDR15}. The main difference to our deviation inequalities is that we take into account the local eigenvalue structure. For instance, from Propositions \ref{RightDev} and \ref{LeftDev}, we get the following theorem:
\begin{thm}\label{ThmEvDev} Grant Assumption \ref{SubGauss}. Then there is a constant $c>0$ depending only on $C_1$ such that for all $y>0$ satisfying
\[
\frac{1}{n(y\wedge 1)}\sum_{k> d}\frac{\lambda_k}{\lambda_d-\lambda_k+y\lambda_d}\le 1/(2C_3^2)
\]
we have
\begin{equation*}
\mathbb{P}\big(  \hat{\lambda}_{d}-\lambda_d> y\lambda_d\big)\le e^{ 1-cn (y\wedge y^2) }.
\end{equation*}
Moreover,  for all $y>0$ satisfying
\begin{equation*}
\frac{1}{n(y\wedge 1)}\sum_{j< d}\frac{\lambda_j}{\lambda_j-\lambda_d+y\lambda_d}\le 1/(2C_3^2)
\end{equation*}
we have
\begin{equation*}
\mathbb{P}\big(  \hat{\lambda}_{d}-\lambda_d< - y\lambda_d\big)\le e^{ 1-cn (y\wedge y^2) }.
\end{equation*}
\end{thm}

If $\lambda_{d}$ is a simple eigenvalue, then Theorem \ref{ThmEvDev} can be seen as a non-asymptotic version of the classical central limit theorem $\sqrt{n}(\hat{\lambda}_{d}/\lambda_{d}-1)\rightarrow\mathcal{N}(0,2)$ which holds for $X$ Gaussian, compare Anderson \cite[Theorem 13.5.1]{And} and Dauxois, Pousse and Romain \cite[Proposition 8]{DPR82}. Moreover, the conditions imposed are related to $\mathbb{E}[ \hat{\lambda}_{d}]$ by the following asymptotic expansion (see e.g. \cite[Equation (2.22)]{N08})
\[
\mathbb{E}[ \hat{\lambda}_{d}/\lambda_{d}]-1=\frac{1}{n}\sum_{k\neq d}\frac{\lambda_k}{\lambda_{d}-\lambda_{k}}+\dots.
\]
A discussion how the eigenvalue conditions in Theorem \ref{ThmEvDev} improve upon standard conditions from the literature is given in Remark  \ref{RemRelativeEVCond}. 

\section{Main tools}\label{SecMainTools}

\subsection{Projector-based calculus}\label{SecProj}

In this section, we present two perturbation formulas, which  together with the representation of the excess risk given in Lemma \ref{SpectralDec} form the basis of our analysis of the excess risk.

\begin{lemma}\label{FirstOrdExp}
For $j\le d$ we have
\begin{equation*}
\Vert P_{j}\hat{P}_{> d}\Vert_2^2 =\sum_{k>d}\frac{\Vert P_j\Delta\hat{P}_k\Vert_2^2}{(\lambda_j-\hat{\lambda}_k)^2},
\end{equation*}
and for $k>d$ we have
\begin{equation*}
 \Vert P_{k}\hat{P}_{\le d}\Vert_2^2 =\sum_{j\le d}\frac{\Vert P_k\Delta\hat{P}_j\Vert_2^2}{(\hat{\lambda}_j-\lambda_k)^2}.
\end{equation*}
Both identities hold provided that all denominators are non-zero.
\end{lemma}

\begin{proof}
The main ingredient is the  formula
\begin{equation}\label{EqBasicBuild}
P_j\hat{P}_k=\frac{1}{\lambda_j-\hat{\lambda}_k}P_j\Delta \hat{P}_k,
\end{equation}
which follows from inserting the spectral representations of $\Sigma$ and $\hat{\Sigma}$ into the right-hand side. Indeed,
\begin{align*}
P_j\Delta \hat{P}_k=\sum_{l\ge 1}\lambda_l P_jP_l\hat{P}_k-\sum_{l\ge 1}\hat{\lambda}_lP_j\hat{P}_l\hat{P}_k=(\lambda_j-\hat{\lambda}_k)P_j\hat{P}_k.
\end{align*}
The first claim now follows from inserting  \eqref{EqBasicBuild} into the identity
\begin{equation*}
\| P_{j}\hat{P}_{> d}\|_2^2=\sum_{k>d}\| P_{j}\hat{P}_k\|_2^2.
\end{equation*}
The second claim follows similarly by switching $j$ and $k$ and summation over~$j$.
\end{proof}

Identity \eqref{EqBasicBuild} can be seen as a basic building block to derive expansions for empirical spectral projectors. Indeed, using \eqref{EqBasicBuild}, we get
\begin{equation}\label{EqFirstOrdExpBB}
P_j\hat{P}_{>d}=\sum_{k>d}\frac{P_j\Delta \hat{P}_k}{\lambda_j-\hat{\lambda}_k}
\end{equation}
and a similar formula for $P_k\hat P_{\le d}$, leading to
\begin{equation}\label{EqFirstOrdExp}
\hat{P}_{> d}-P_{> d}=P_{\le d}\hat{P}_{> d}-P_{> d}\hat{P}_{\le d}=\sum_{j\le d}\sum_{k>d}\bigg(\frac{P_j\Delta \hat{P}_k}{\lambda_j-\hat{\lambda}_k}+\frac{P_k\Delta \hat{P}_j}{\hat{\lambda}_j-\lambda_k}\bigg).
\end{equation}
The following lemma immediately leads to a linear expansion of $\hat{P}_{> d}$.
\begin{lemma}\label{LemLinExp} For $j\le d$ we have
\begin{align*}
P_j\hat{P}_{>d}&=\sum_{k>d}\frac{P_j\Delta P_k}{\lambda_j-\lambda_k}+\sum_{k\le d}\sum_{l>d}\frac{P_j\Delta P_k\Delta\hat{P}_l}{(\lambda_j-\hat{\lambda}_l)(\lambda_k-\hat{\lambda}_l)}\\&+\sum_{k>d}\sum_{l\le d}\frac{P_j\Delta P_k\Delta\hat{P}_l}{(\lambda_j-\lambda_k)(\hat{\lambda}_l-\lambda_k)}-\sum_{k>d}\sum_{l>d}\frac{P_j\Delta P_k\Delta \hat{P}_l}{(\lambda_j-\hat{\lambda}_l)(\lambda_j-\lambda_k)}
\end{align*}
and for $k>d$ we have
\begin{align*}
P_k\hat{P}_{\le d}&=\sum_{j\le d}\frac{P_k\Delta P_j}{\lambda_k-\lambda_j}+\sum_{j>d}\sum_{l\le d}\frac{P_k\Delta P_j\Delta\hat{P}_l}{(\lambda_k-\hat{\lambda}_l)(\lambda_j-\hat{\lambda}_l)}\\
&+\sum_{j\le d}\sum_{l>d}\frac{P_k\Delta P_j\Delta\hat{P}_l}{(\lambda_k-\lambda_j)(\hat{\lambda}_l-\lambda_j)}-\sum_{j\le d}\sum_{l \le d}\frac{P_k\Delta P_j\Delta \hat{P}_l}{(\lambda_k-\lambda_j)(\lambda_k-\hat{\lambda}_l)}.
\end{align*}
Both identities hold provided that all denominators are non-zero.
\end{lemma}

\begin{proof}
We only prove the first identity, since the second one follows by the same line of arguments.
First using \eqref{EqFirstOrdExpBB} and the identity $I=P_{\le d}+P_{> d}=\hat{P}_{\le d}+\hat{P}_{> d}$, we have
\begin{equation*}
P_j\hat{P}_{>d}=\sum_{l>d}\frac{P_j\Delta \hat{P}_l}{\lambda_j-\hat{\lambda}_l}=\sum_{l>d}\frac{P_j\Delta P_{\le d}\hat{P}_l}{\lambda_j-\hat{\lambda}_l}+\sum_{l>d}\frac{P_j\Delta P_{>d}\hat{P}_l}{\lambda_j-\hat{\lambda}_l}
\end{equation*}
and
\begin{equation*}
\sum_{k>d}\frac{P_j\Delta P_k}{\lambda_j-\lambda_k}=\sum_{k>d}\frac{P_j\Delta P_k\hat{P}_{\le d}}{\lambda_j-\lambda_k}+\sum_{k>d}\frac{P_j\Delta P_k\hat{P}_{> d}}{\lambda_j-\lambda_k}.
\end{equation*}
Thus
\begin{align}
P_j\hat{P}_{>d}&=\sum_{k>d}\frac{P_j\Delta P_k}{\lambda_j-\lambda_k}
+\sum_{l>d}\frac{P_j\Delta P_{\le d}\hat{P}_l}{\lambda_j-\hat{\lambda}_l}
-\sum_{k>d}\frac{P_j\Delta P_k\hat{P}_{\le d}}{\lambda_j-\lambda_k}\label{eq:d1}\\
&+\Bigg( \sum_{l>d}\frac{P_j\Delta P_{>d}\hat{P}_l}{\lambda_j-\hat{\lambda}_l}-\sum_{k>d}\frac{P_j\Delta P_k\hat{P}_{> d}}{\lambda_j-\lambda_k}\Bigg)\nonumber.
\end{align}
Using \eqref{EqBasicBuild}, we get
\begin{equation*}
\sum_{l>d}\frac{P_j\Delta P_{\le d}\hat{P}_l}{\lambda_j-\hat{\lambda}_l}=\sum_{k\le d}\sum_{l>d}\frac{P_j\Delta P_k\Delta\hat{P}_l}{(\lambda_j-\hat{\lambda}_l)(\lambda_k-\hat{\lambda}_l)}
\end{equation*}
and
\begin{equation*}
-\sum_{k>d}\frac{P_j\Delta P_k\hat{P}_{\le d}}{\lambda_j-\lambda_k}=\sum_{k>d}\sum_{l\le d}\frac{P_j\Delta P_k\Delta\hat{P}_l}{(\lambda_j-\lambda_k)(\hat{\lambda}_l-\lambda_k)}.
\end{equation*}
Moreover, again using \eqref{EqBasicBuild}, the term in brackets in \eqref{eq:d1} is equal to
\begin{align*}
&\sum_{l>d}\frac{P_j\Delta P_{> d}\hat{P}_l}{\lambda_j-\hat{\lambda}_l}-\sum_{k>d}\frac{P_j\Delta P_k\hat{P}_{>d}}{\lambda_j-\lambda_k}\\
&=-\sum_{k>d}\sum_{l>d} \frac{\lambda_k-\hat{\lambda}_l}{(\lambda_j-\hat{\lambda}_l)(\lambda_j-\lambda_k)} P_j\Delta P_k \hat{P}_l\\
&=-\sum_{k>d}\sum_{l>d}\frac{1}{(\lambda_j-\hat{\lambda}_l)(\lambda_j-\lambda_k)} P_j\Delta P_k\Delta \hat{P}_l,
\end{align*}
and the claim follows.
\end{proof}

\begin{remark}\label{RemPT1}  Note that compared to \eqref{EqFirstOrdExpBB}, where only spectral gaps between $j$ and $k>d$ appear, the first formula in Lemma \ref{LemLinExp} includes all spectral gaps between $k> d$ and $l\le d$, even in the case $j=1$. Since we are also interested in the case of small spectral gaps (including $\lambda_d=\lambda_{d+1}$), our main analysis of the excess risk will be based on Lemma \ref{FirstOrdExp}. Lemma \ref{LemLinExp} will be important to derive linear expansions for the excess risk under stronger eigenvalue conditions.
\end{remark}
\begin{remark}\label{RemPT2}
Usually, expansions for spectral projectors are obtained by the Cauchy integral representation for spectral projectors in combination with the second resolvent equation (resp. the second Neumann series), see e.g. Kato \cite{K95}. The difference of Lemmas \ref{FirstOrdExp} and \ref{LemLinExp} to the formulation in e.g. \cite[Lemma 2]{KL15a} or \cite[Theorem 5.1.4]{HE15} is the form of the remainder term. In \cite{KL15a,HE15} the remainder term is given by an integral over the resolvent, while the above results lead to an algebraic form of the remainder term. In Section \ref{SecErrorDec} and Appendix \ref{SecLinExp} we will use these algebraic expressions to establish recursion arguments.
\end{remark}

\subsection{Error decompositions} \label{SecErrorDec}
In this section, we prove deterministic upper bounds for the excess risk which form the basis of our new upper bounds in Section \ref{SecNewBounds}. For ${\mathcal E}^{PCA}_{\le d}(\mu)$ we split the sum into indices $j\le r$, where we expect the spectral gaps $\lambda_j-\lambda_{d+1}$ to be large, meaning that we can insert the perturbation formulas from Lemma \ref{FirstOrdExp}, and into indices $r<j\le d$, where we expect the spectral gaps $\lambda_j-\mu$ to be small, meaning that wrong projections do not incur a large error. The terms of the first sum can then be controlled by a recursion argument.

\begin{proposition}\label{ErrorDecR} For $\mu\in [\lambda_{d+1},\lambda_d]$ and $r=0,\dots,d$, we have
\begin{align}
{\mathcal E}^{PCA}_{\le d}(\mu)
\le & 4\sum_{j\le r}(\lambda_j-\mu)\frac{\Vert P_{j}\Delta\hat{P}_{>d}\Vert_2^2}{(\lambda_j-\lambda_{d+1})^2}+\sum_{j=r+1}^{d\wedge (r+p-d)}(\lambda_j-\mu)\label{EqErrorDecR1}\\
&+ \sum_{j\le r}(\lambda_j-\mu) \mathbbm{1}\big(\hat{\lambda}_{d+1}-\lambda_{d+1}> (\lambda_j-\lambda_{d+1})/2\big)\nonumber.
\end{align}
Furthermore, for $s=r,\ldots,d$ and the weighted projector
\begin{equation}\label{EqWeightOper}
S_{\le s}=S_{\le s}(\mu)=\sum_{j\le s} \frac{1}{\sqrt{\lambda_j-\mu}}P_j
\end{equation}
(assuming $\lambda_{s}>\mu$) we obtain
\begin{align}
{\mathcal E}^{PCA}_{\le d}(\mu)
\le & 16\sum_{j\le r}(\lambda_j-\mu)\frac{\Vert P_{j}\Delta P_{>s}\Vert_2^2}{(\lambda_j-\lambda_{d+1})^2}+2\sum_{j=r+1}^{d\wedge (r+p-d)}(\lambda_j-\mu)\label{EqErrorDecR2}\\
&+ 2\sum_{j\le r}(\lambda_j-\mu) \mathbbm{1}\big(\hat{\lambda}_{d+1}-\lambda_{d+1}> (\lambda_j-\lambda_{d+1})/2\big)\nonumber\\
&+8\sum_{j\le r}(\lambda_j-\mu)\frac{\Vert P_{j}\Delta\Vert_2^2}{(\lambda_j-\lambda_{d+1})^2} \mathbbm{1}\big(\|S_{\le s}\Delta S_{\le s}\|_\infty > 1/4\big)\nonumber.
\end{align}
\end{proposition}

\begin{remark}\label{RemConvChoice2}
Note that the convention of Remark \ref{RemConvChoice1} is still in force. For certain values of $r$ and $s$, the upper bounds in Proposition \ref{ErrorDecR} may depend on the choice of the $P_j$. The actual choices, however, do not alter the final results in Section \ref{SecNewBounds}.
\end{remark}

\begin{remark}\label{RemConstants}
The constants are chosen for simplicity. For each $\eps>0$, the constant $16$ in \eqref{EqErrorDecR2} can be replaced by $1+\eps$ provided that the constants $1/2$ and $1/4$ in the definition of the events are replaced by bigger constants depending on $\eps$.
\end{remark}

\begin{proof}
Using $\Vert P_{j}\hat{P}_{> d}\Vert_2^2\le 1$ and $\sum_{j=r+1}^d\Vert P_{j}\hat{P}_{> d}\Vert_2^2\le p-d$,
we obtain
\begin{equation}\label{EqTrivBound}
{\mathcal E}^{PCA}_{\le d}(\mu)\le \sum_{j\le r}(\lambda_j-\mu)\Vert P_{j}\hat{P}_{> d}\Vert_2^2+\sum_{j=r+1}^{d\wedge (r+p-d)}(\lambda_j-\mu).
\end{equation}
By Lemma \ref{FirstOrdExp}, we have
\begin{equation*}
\Vert P_{j}\hat{P}_{> d}\Vert_2^2 =\sum_{k> d}\frac{\Vert P_j\Delta\hat{P}_k\Vert_2^2}{(\lambda_j-\hat{\lambda}_k)^2}.
\end{equation*}
Moreover, on the event
\[
\lbrace \hat{\lambda}_{d+1}-\lambda_{d+1} \le  (\lambda_j-\lambda_{d+1})/2\rbrace=\lbrace\lambda_j- \hat{\lambda}_{d+1} \ge (\lambda_j-\lambda_{d+1})/2\rbrace
\]
we can bound
\begin{equation}\label{EqKeyP}
\Vert P_{j}\hat{P}_{> d}\Vert_2^2\le \sum_{k>d}4\frac{\Vert P_{j}\Delta\hat{P}_{k}\Vert_2^2}{(\lambda_j-\lambda_{d+1})^2}= 4\frac{\Vert P_{j}\Delta\hat{P}_{ > d}\Vert_2^2}{(\lambda_j-\lambda_{d+1})^2}.
\end{equation}
By \eqref{EqKeyP} and $\Vert P_{j}\hat{P}_{> d}\Vert_2^2\le 1$, we conclude that
\begin{equation}\label{EqExpBound}
\Vert P_{j}\hat{P}_{> d}\Vert_2^2\le 4\frac{\Vert P_{j}\Delta\hat{P}_{ > d}\Vert_2^2}{(\lambda_j-\lambda_{d+1})^2}+\mathbbm{1}\big( \hat{\lambda}_{d+1}-\lambda_{d+1} > (\lambda_j-\lambda_{d+1})/2\big).
\end{equation}
Inserting \eqref{EqExpBound} into \eqref{EqTrivBound}, we obtain the first claim \eqref{EqErrorDecR1}. The second claim follows from an additional recursion argument. For this, we introduce 
\begin{equation}\label{EqWeightOperR}
R_{\le s}=R_{\le s}(\mu)=\sum_{j\le s} \sqrt{\lambda_j-\mu}P_j,
\end{equation}
which satisfies the identities $S_{\le s}R_{\le s}=P_{\le s}$ and
\begin{equation}\label{EqExcRiskRepr}
\sum_{j\le s}(\lambda_j-\mu)\Vert P_j\hat{P}_{>d}\Vert_2^2=\Vert R_{\le s}\hat{P}_{>d}\Vert_2^2.
\end{equation}
Then we have
\begin{align}
&\sum_{j\le r} (\lambda_j-\mu)\frac{\Vert P_{j}\Delta\hat{P}_{ > d}\Vert_2^2}{(\lambda_j-\lambda_{d+1})^2}\label{EqRecBound1}\\
&\le 2\sum_{j\le r} (\lambda_j-\mu)\frac{\Vert P_{j}\Delta P_{>s}\hat{P}_{ > d}\Vert_2^2}{(\lambda_j-\lambda_{d+1})^2}+2\sum_{j\le r} (\lambda_j-\mu)\frac{\Vert P_{j}\Delta P_{\le s}\hat{P}_{ > d}\Vert_2^2}{(\lambda_j-\lambda_{d+1})^2}\nonumber\\
&\le 2\sum_{j\le r} (\lambda_j-\mu)\frac{\Vert P_{j}\Delta P_{>s}\Vert_2^2}{(\lambda_j-\lambda_{d+1})^2}+2\sum_{j\le r} \frac{\Vert P_{j}\Delta P_{\le s}\hat{P}_{ > d}\Vert_2^2}{\lambda_j-\mu}\nonumber\\
&=2\sum_{j\le r} (\lambda_j-\mu)\frac{\Vert P_{j}\Delta P_{>s}\Vert_2^2}{(\lambda_j-\lambda_{d+1})^2}+2 \Vert S_{\le r}\Delta P_{\le s}\hat{P}_{ > d}\Vert_2^2\nonumber.
\end{align}
On the event $\{\|S_{\le s}\Delta S_{\le s}\|_\infty \le 1/4\}$, the last term is bounded via
\begin{align*}
2\Vert S_{\le r}\Delta P_{\le s}\hat{P}_{ > d}\Vert_2^2&=2\Vert S_{\le r}\Delta S_{\le s}R_{\le s}\hat{P}_{ > d}\Vert_2^2\\
&\le 2\Vert S_{\le r}\Delta S_{\le s}\Vert_\infty^2\Vert R_{\le s}\hat{P}_{ > d}\Vert_2^2\\
&\le 2\Vert S_{\le s}\Delta S_{\le s}\Vert_\infty^2\Vert R_{\le s}\hat{P}_{ > d}\Vert_2^2\le \Vert R_{\le s}\hat{P}_{ > d}\Vert_2^2/8,
\end{align*}
where we also used that $r\le s$. Thus, on $\{\|S_{\le s}\Delta S_{\le s}\|_\infty \le 1/4\}$, we get
\begin{align}
&\sum_{j\le r} (\lambda_j-\mu)\frac{\Vert P_{j}\Delta\hat{P}_{ > d}\Vert_2^2}{(\lambda_j-\lambda_{d+1})^2}\label{EqRecBound2}\\
&\le 2\sum_{j\le r} (\lambda_j-\mu) \frac{\Vert P_{j}\Delta P_{ > s}\Vert_2^2}{(\lambda_j-\lambda_{d+1})^2} +\frac18\sum_{j\le s}(\lambda_j-\mu)\Vert P_{j}\hat{P}_{> d}\Vert_2^2\nonumber.
\end{align}
Using also that $\Vert P_{j}\Delta\hat{P}_{ > d}\Vert_2^2\le \Vert P_{j}\Delta\Vert_2^2$, we conclude that
\begin{align*}
&4\sum_{j\le r} (\lambda_j-\mu)\frac{\Vert P_{j}\Delta\hat{P}_{ > d}\Vert_2^2}{(\lambda_j-\lambda_{d+1})^2}\\
&\le 8\sum_{j\le r} (\lambda_j-\mu) \frac{\Vert P_{j}\Delta P_{ > s}\Vert_2^2}{(\lambda_j-\lambda_{d+1})^2}+\frac12 {\mathcal E}^{PCA}_{\le d}(\mu)\\
&+4\sum_{j\le r}(\lambda_j-\mu)\frac{\Vert P_{j}\Delta \Vert_2^2}{(\lambda_j-\lambda_{d+1})^2}  \mathbbm{1}\big(\|S_{\le s}\Delta S_{\le s}\|_\infty > 1/4\big).
\end{align*}
Plugging this into \eqref{EqErrorDecR1}, we obtain the second claim.
\end{proof}

Similarly, we can upper-bound the second risk part ${\mathcal E}^{PCA}_{> d}$. The only difference in the proof is that an additional argument deals with the sum over all sufficiently large $k$.

\begin{proposition}\label{ErrorDecL}
 For $\mu\in [\lambda_{d+1},\lambda_d]$ and $l=d+1,\dots,p+1$, we have
\begin{align}
{\mathcal E}^{PCA}_{> d}(\mu)
\le & 4\sum_{k\ge  l}(\mu-\lambda_{k})\frac{\Vert P_{k}\Delta \hat{P}_{\le d}\Vert_2^2}{(\lambda_d-\lambda_k)^2}+\sum_{k=(d+1)\vee(l-d)}^{l-1}(\mu-\lambda_k)\label{EqErrorDecL1}\\
&+\sum_{\substack{k\ge  l:\\ \lambda_k\ge \lambda_d/2}}(\mu-\lambda_k) \mathbbm{1}\big(\hat{\lambda}_{d}-\lambda_d<-  (\lambda_d-\lambda_{k})/2\big)\nonumber\\
&+d(\mu-\lambda_p) \mathbbm{1}\big(\hat{\lambda}_{d}-\lambda_{d}<-  \lambda_d/4\big).\nonumber
\end{align}
\end{proposition}
Note that for $p=\infty$ the convention of Proposition \ref{PropMainRes1} is still in force.

\begin{proof}
Using $\sum_{d<k< l}\Vert P_k\hat{P}_{\le d}\Vert_2^2\le d$ and $\Vert P_k\hat{P}_{\le d}\Vert_2^2\le 1$, we obtain
\begin{align}
{\mathcal E}^{PCA}_{> d}(\mu)\le \sum_{k\ge l}(\mu-\lambda_{k})\Vert P_k\hat{P}_{\le d}\Vert_2^2+\sum_{k=(d+1)\vee(l-d)}^{l-1}(\mu-\lambda_k)\label{EqTrivBound2}.
\end{align}
Proceeding as in the proof of Proposition \ref{ErrorDecR},  on the event
\begin{equation*}
\lbrace \hat{\lambda}_d-\lambda_d\ge - (\lambda_d-\lambda_k)/2\rbrace=\lbrace\hat{\lambda}_d-\lambda_k\ge  (\lambda_d-\lambda_k)/2\rbrace,
\end{equation*}
we have
\begin{equation}\label{EqKeyP2}
\Vert P_k\hat{P}_{\le d}\Vert_2^2\le 4\frac{\Vert P_{k}\Delta \hat{P}_{\le d}\Vert_2^2}{(\lambda_d-\lambda_k)^2}.
\end{equation}
Using $\Vert P_k\hat{P}_{\le d}\Vert_2^2\le 1$, we get
\begin{align}
\Vert P_k\hat{P}_{\le d}\Vert_2^2
\le 4\frac{\Vert P_k\Delta  \hat{P}_{\le d}\Vert_2^2}{(\lambda_d-\lambda_k)^2}+\mathbbm{1}\big(\hat{\lambda}_{d}-\lambda_{d}<- (\lambda_d-\lambda_k)/2\big)\label{EqLocExp}.
\end{align}
For $k\ge l$ such that $\lambda_k< \lambda_d/2$, we have
\begin{equation*}
\lbrace \hat{\lambda}_d-\lambda_d< -(\lambda_d-\lambda_k)/2\rbrace\subseteq \lbrace\hat{\lambda}_d-\lambda_d< -\lambda_d/4\rbrace.
\end{equation*}
Hence, by \eqref{EqKeyP2} and the bound
\[
\sum_{\substack{k\ge l:\\\lambda_k< \lambda_d/2}}(\mu-\lambda_k)\Vert P_k\hat{P}_{\le d}\Vert_2^2\le d(\mu-\lambda_p),
\]
we have
\begin{align}
&\sum_{\substack{k\ge l:\\\lambda_k< \lambda_d/2}}(\mu-\lambda_k)\Vert P_k\hat{P}_{\le d}\Vert_2^2\label{EqLocExp2}\\
&\le 4\sum_{\substack{k\ge l:\\\lambda_k< \lambda_d/2}}(\mu-\lambda_k)\frac{\Vert P_k\Delta\hat{P}_{\le d} \Vert_2^2}{(\lambda_d-\lambda_k)^2}+d(\mu-\lambda_p)  \mathbbm{1}\big(\hat{\lambda}_d-\lambda_d< -\lambda_d/4\big)\nonumber.
\end{align}
Inserting \eqref{EqLocExp} (for $k\ge l$ such that $\lambda_k\ge \lambda_d/2$) and \eqref{EqLocExp2} into \eqref{EqTrivBound2}, the claim follows.
\end{proof}

\subsection{Concentration inequalities}\label{SecConcIneq}
In order to make the deterministic upper bounds of the previous section useful, one has to show that the events in the remainder terms occur with small probability. We establish concentration inequalities for the weighted sample covariance operators as well as deviation inequalities for the empirical eigenvalues $\hat{\lambda}_{d}$ and $\hat{\lambda}_{d+1}$, based on the concentration inequality \cite[Corollary 2]{KL14} for sample covariance operators which we use in the form
\begin{equation}\label{EqKolLou}
\mathbb{P}( \Vert \Delta\Vert_\infty>C_3\lambda_1 x)\le e^{-n(x\wedge x^2)},
\end{equation}
whenever
\[
\tr(\Sigma)\le n\lambda_1(x\wedge x^2),
\]
where $C_3>1$ is a constant which depends only on $C_1$.
First, consider the weighted projector $S_{\le s}$ from \eqref{EqWeightOper} for $\mu\in[0,\lambda_{s})$. Then, as in \eqref{EqSigma'},
$X'=S_{\le s}X$ satisfies Assumption \ref{SubGauss} with the same constant $C_1$ as $X$ and has covariance operator
\[\Sigma'=S_{\le s}\Sigma S_{\le s}=\sum_{j\le s} \frac{\lambda_j}{\lambda_j-\mu}P_j.
\]
The eigenvalues of $\Sigma'$ (in decreasing order) are
$\lambda_j'=\lambda_{s+1-j}/(\lambda_{s+1-j}-\mu)$,
noting that the order is reversed by the weighting. Using the sample covariance $\hat{\Sigma}'=S_{\le s}\hat{\Sigma} S_{\le s}$ and choosing
$x=1/(4C_3\lambda_1')$,
which is smaller than $1$, the concentration inequality \eqref{EqKolLou} applied to $\Delta'=\Sigma'-\hat\Sigma'$ yields:

\begin{lemma}\label{LemWeightedCov} Grant Assumption \ref{SubGauss}. If  $\mu\in[0,\lambda_{s})$ and if
\begin{equation*}
\frac{\lambda_{s}}{\lambda_{s}-\mu}\sum_{j\le s}\frac{\lambda_j}{\lambda_j-\mu}\le n/(16C_3^2)
\end{equation*}
holds with  the constant $C_3$ from \eqref{EqKolLou}, then
\begin{equation*}
\mathbb{P}\big( \Vert S_{\le s}\Delta S_{\le s}\Vert_\infty>1/4\big)\le \exp\bigg( -\frac{n(\lambda_{s}-\mu)^2}{16C_3^2\lambda^2_{s}}\bigg).
\end{equation*}
\end{lemma}

Next, we will state deviation inequalities for the empirical eigenvalues $\hat{\lambda}_{d}$ and $\hat{\lambda}_{d+1}$, namely right-deviation inequalities for  $\hat{\lambda}_{d+1}$ and left-deviation inequalities for $\hat{\lambda}_{d}$.

\begin{proposition}\label{RightDev} Grant Assumption \ref{SubGauss}. For all $x> 0$ satisfying
\begin{equation}\label{EqRightDev}
\max\bigg(\frac{C_3\lambda_{d+1}}{x},1\bigg)\sum_{k>d}\frac{\lambda_k}{\lambda_{d+1}-\lambda_k+x}\le n/C_3,
\end{equation}
we have
\begin{equation*}
\mathbb{P}\big( \hat{\lambda}_{d+1}-\lambda_{d+1}>x\big)\le \exp\bigg( -n\min\bigg(\frac{x^2}{C_3^2\lambda_{d+1}^2}, \frac{x}{C_3\lambda_{d+1}}\bigg) \bigg),
\end{equation*}
where $C_3$ is the constant in \eqref{EqKolLou}.
\end{proposition}

\begin{proof}
First, we apply the min-max characterisation of eigenvalues and obtain $\hat{\lambda}_{d+1}\le \lambda_1(P_{> d}\hat{\Sigma} P_{> d})$.
This gives
\begin{equation}\label{EqRightDev1}
\mathbb{P}( \hat{\lambda}_{d+1}-\lambda_{d+1}> x)\le \mathbb{P}( \lambda_1(P_{>d}\hat{\Sigma} P_{>d})-\lambda_1(P_{>d}\Sigma P_{>d})> x).
\end{equation}
We now use the following lemma, proven later.

\begin{lemma}\label{ertz}
Let $S$ and $T$ be self-adjoint, positive compact operators on $\mathcal{H}$ and $y>\lambda_1(S)$. Then:
 \begin{align*}
\lambda_1(T)> y &\iff \lambda_1((y-S)^{-1/2}(T-S)(y-S)^{-1/2})> 1.
\end{align*}
\end{lemma}

Applying this lemma to $S=P_{>d}\Sigma P_{>d}$, $T=P_{>d}\hat{\Sigma} P_{>d}$, and $y=\lambda_1(S)+x=\lambda_{d+1}+x$, we get
\begin{equation}\label{EqRightDev2}
\mathbb{P}( \lambda_1(P_{>d}\hat{\Sigma} P_{>d})-\lambda_1(P_{>d}\Sigma P_{>d})> x)\le \mathbb{P}( \|T_{>d}\Delta T_{> d}\|_\infty>1)
\end{equation}
with
\begin{equation*}
T_{> d}=\sum_{k>d}\frac{1}{\sqrt{\lambda_{d+1}-\lambda_k+x}}P_k.
\end{equation*}
Thus, as in \eqref{EqSigma'}, we consider $X'=T_{> d}X$, satisfying  Assumption \ref{SubGauss} with the same constant $C_1$, and obtain the covariance operator
\[\Sigma'=T_{> d}\Sigma T_{> d}=\sum_{k>d} \frac{\lambda_k}{\lambda_{d+1}-\lambda_{k}+x}P_k.
\]
Hence choosing
\begin{equation*}
x'=\frac{1}{C_3\lambda_1'}=\frac{x}{C_3\lambda_{d+1}},
\end{equation*}
the concentration inequality \eqref{EqKolLou}, applied to $\Delta'=T_{>d}\Delta T_{>d}$ and $x'$, gives
\begin{equation}\label{EqRightDev3}
\mathbb{P}( \|T_{> d}\Delta T_{> d}\|_\infty>1)\le \exp\bigg( -n\min\bigg(\frac{x^2}{C_3^2\lambda_{d+1}^2}, \frac{x}{C_3\lambda_{d+1}}\bigg) \bigg)
\end{equation}
in view of Condition \eqref{EqRightDev}. Combining \eqref{EqRightDev1}-\eqref{EqRightDev3}, the claim follows.

It remains to prove Lemma \ref{ertz}.
We have
\[
\lambda_1\big( (y-S)^{-1/2}(T-S)(y-S)^{-1/2}\big) \le 1
\]
if and only if (for a linear operator $L:\mathcal{H}\rightarrow \mathcal{H}$, we write $L\ge 0$ if $L$ is positive, i.e. if $\langle Lx,x\rangle\ge 0$ for all $x\in \mathcal{H}$)
\begin{equation*}
(y-S)^{-1/2}(y-T)(y-S)^{-1/2}=I-(y-S)^{-1/2}(T-S)(y-S)^{-1/2}\ge 0.
\end{equation*}
Since $(y-S)^{-1/2}$ is self-adjoint and strictly positive, this is the case if and only if $y-T\ge 0$, that is $\lambda_1(T)\le y$. A logical negation yields the assertion of the lemma.
\end{proof}

In view of the error decompositions \eqref{EqErrorDecR1} and \eqref{EqErrorDecR2}, we want to apply Proposition \ref{RightDev} with $x=(\lambda_j-\lambda_{d+1})/2$, $j\le d$. For this, we require
\begin{equation*}
\max\bigg(\frac{2C_3\lambda_{d+1}}{\lambda_j-\lambda_{d+1}},1\bigg)\sum_{k>d}\frac{\lambda_k}{\lambda_{j}-\lambda_k}\le n/(2C_3).
\end{equation*}
Simplifying the maximum yields:

\begin{corollary}\label{CorRightDev} Grant Assumption \ref{SubGauss} and let $j\le d$. Suppose that
\begin{equation}
\frac{\lambda_{j}}{\lambda_j-\lambda_{d+1}}\sum_{k>d}\frac{\lambda_k}{\lambda_{j}-\lambda_k}\le n/(4C_3^2).\label{CondCorRightDev}
\end{equation}
Then
\begin{equation}
\mathbb{P}\Big( \hat{\lambda}_{d+1}-\lambda_{d+1}> \frac{\lambda_j-\lambda_{d+1}}{2}\Big)
 \le \exp\Big(-\frac{n(\lambda_j-\lambda_{d+1})^2}{4C_3^2\lambda_{j}^2} \Big). \label{EqCorRightDev}
\end{equation}
\end{corollary}

The corresponding left-deviation result for $\hat\lambda_{d}$ is as follows.

\begin{proposition}\label{LeftDev}  Grant Assumption \ref{SubGauss}. For all $x> 0$ satisfying
\begin{equation}\label{EqLeftDev}
\max\bigg(\frac{C_3\lambda_d}{x},1\bigg)\sum_{j\le d}\frac{\lambda_j}{\lambda_j-\lambda_d+x}\le n/C_3,
\end{equation}
we have
\begin{equation*}
\mathbb{P}\big( \hat{\lambda}_{d}-\lambda_{d}<-x\big)\le \exp\bigg( -n\min\bigg(\frac{x^2}{C_3^2\lambda_d^2}, \frac{x}{C_3\lambda_d}\bigg) \bigg),
\end{equation*}
where $C_3$ is the constant in \eqref{EqKolLou}.
\end{proposition}

\begin{proof}
First, we apply the max-min characterisation of eigenvalues and obtain
$\hat{\lambda}_{d}\ge \lambda_d(P_{\le d}\hat{\Sigma} P_{\le d})$.
This gives
\begin{equation}\label{EqLeftDev1}
\mathbb{P}( \hat{\lambda}_{d}-\lambda_{d}<-x)\le \mathbb{P}( \lambda_d(P_{\le d}\hat{\Sigma} P_{\le d})-\lambda_d(P_{\le d}\Sigma P_{\le d})<-x).
\end{equation}

Similar to Lemma \ref{ertz}, we have for self-adjoint, positive  operators $S,T$ on $V_d=\operatorname{span}(u_1,\dots,u_d)$  and $y<\lambda_d(S)$
\[
\lambda_1\big( (S-y)^{-1/2}(S-T)(S-y)^{-1/2}\big) \le 1
\]
if and only if (for the operator partial ordering)
\begin{equation*}
(S-y)^{-1/2}(T-y)(S-y)^{-1/2}=I-(S-y)^{-1/2}(S-T)(S-y)^{-1/2}\ge  0.
\end{equation*}
This is the case if and only if $T-y\ge 0$, that is $\lambda_d(T)\ge y$.

Applying the negation of this equivalence to $S=P_{\le d}\Sigma P_{\le d}$, $T=P_{\le d}\hat{\Sigma} P_{\le d}$, and $y=\lambda_d(S)-x=\lambda_{d}-x$, we get
\begin{equation}\label{EqLeftDev2}
\mathbb{P}( \lambda_d(P_{\le d}\hat{\Sigma} P_{\le d})-\lambda_d(P_{\le d}\Sigma P_{\le d})<- x)\le \mathbb{P}( \|T_{\le d}\Delta T_{\le d}\|_\infty>1)
\end{equation}
with
\begin{equation*}
T_{\le d}=\sum_{j\le d}\frac{1}{\sqrt{\lambda_j-\lambda_d+x}}P_j.
\end{equation*}
We consider $X'=T_{\le d}X$, satisfying  Assumption \ref{SubGauss} with the same constant $C_1$, and obtain the covariance operator
\[\Sigma'=T_{\le d}\Sigma T_{\le d}=\sum_{j\le d} \frac{\lambda_j}{\lambda_{j}-\lambda_{d}+x}P_j.
\]
Hence, noting that $\lambda\mapsto \lambda/(\lambda-\lambda_d+x)$ is decreasing for $\lambda\ge\lambda_d$, we choose
\begin{equation*}
x'=\frac{1}{C_3\lambda_1'}=\frac{x}{C_3\lambda_{d}}.
\end{equation*}
Then \eqref{EqKolLou}, applied to $\Delta'=T_{\le d}\Delta T_{\le d}$ and $x'$, gives
\begin{equation}\label{EqLeftDev3}
\mathbb{P}( \|T_{\le d}\Delta T_{\le d}\|_\infty>1)\le \exp\bigg( -n\min\bigg(\frac{x^2}{C_3^2\lambda_{d}^2}, \frac{x}{C_3\lambda_{d}}\bigg) \bigg)
\end{equation}
in view of Condition \eqref{EqLeftDev}. We conclude by \eqref{EqLeftDev1}-\eqref{EqLeftDev3}.
\end{proof}

In particular, choosing $x=(\lambda_d-\lambda_{k})/2$, we get:
\begin{corollary}\label{CorLeftDev} Grant Assumption \ref{SubGauss} and let $k> d$. Suppose that
\begin{equation}\label{CondCorLeftDev}
\frac{\lambda_d}{\lambda_d-\lambda_{k}}\sum_{j\le d}\frac{\lambda_j}{\lambda_j-\lambda_{k}}\le n/(4C_3^2).
\end{equation}
Then
\begin{equation}\label{EqCorLeftDev}
\mathbb{P}\big( \hat{\lambda}_{d}-\lambda_{d}<-(\lambda_d-\lambda_{k})/2\big) \le \exp\bigg( -\frac{n(\lambda_d-\lambda_{k})^2}{4C_3^2\lambda_d^2}\bigg) .
\end{equation}
\end{corollary}
\begin{remark}\label{RemRelativeEVCond}
Let us consider the important special case $j=s=d$, $k=d+1$, and $\mu=\lambda_{d+1}$. Then all three conditions in Lemma \ref{LemWeightedCov} and Corollaries \ref{CorRightDev}, \ref{CorLeftDev} are implied by
\begin{equation}\label{EqLocalCond}
\frac{\lambda_d}{\lambda_d-\lambda_{d+1}}\bigg( \sum_{j\le d}\frac{\lambda_j}{\lambda_j-\lambda_{d+1}}+\sum_{k>d}\frac{\lambda_k}{\lambda_d-\lambda_{k}}\bigg) \le n/(16C_3^2).
\end{equation}
In particular, if \eqref{EqLocalCond} holds, then all events in the remainder terms in Propositions \ref{ErrorDecR} and \ref{ErrorDecL} occur with small probability.

The localised analysis of this section can be compared to the following absolute one. All events considered in Lemma \ref{LemWeightedCov} and Corollaries \ref{CorRightDev}, \ref{CorLeftDev} are contained in $\{\|\Delta\|_{\infty}>(\lambda_d-\lambda_{d+1})/4\}$ and by \eqref{EqKolLou} this occurs with small probability if
\begin{equation}\label{EqAbsoluteCond}
\frac{\lambda_1\tr(\Sigma)}{(\lambda_d-\lambda_{d+1})^2}\le n/(16C_3^2).
\end{equation}
Note that the condition that $\|\Delta\|_\infty$ is small relative to certain spectral gaps, here $\lambda_d-\lambda_{d+1}$, is often encountered in perturbation theory, see e.g. \cite[Theorem VII.3.1]{B97}, \cite[Theorems 5.1.4 and 5.1.8]{HE15}, and \cite[Lemma 1]{KL15a}. Many of our mathematical issues arise from showing that Condition \eqref{EqAbsoluteCond} can be replaced by the localised version in \eqref{EqLocalCond}.
\end{remark}

\begin{remark}
Our concentration inequalities rely on Assumption \ref{SubGauss}. Generalisations are possible under weaker moment assumptions, including $\sup_{j\ge 1}\mathbb{E}[|\lambda_j^{-1/2}\langle X, u_j\rangle|^k]<\infty$ for some $k> 4$. Since the latter seemingly leads to stronger eigenvalue conditions than formulated in Lemma \ref{LemWeightedCov} and Corollaries \ref{CorRightDev}, \ref{CorLeftDev}, such generalisations are not pursued here.
\end{remark}

\section{Proofs}\label{SecProofs}
We now provide the proofs of the results in Section \ref{SecNewBounds} by combining the error decompositions of Section \ref{SecErrorDec} with the concentration inequalities of Section \ref{SecConcIneq}.

\subsection{Proof of Lemma \ref{SpectralDec}}
Inserting the spectral representation of $\Sigma$, the excess risk can be written as
\begin{equation*}
{\mathcal E}^{PCA}_d=\langle\Sigma, P_{\le d}-\hat{P}_{\le d}\rangle=\sum_{j\ge 1} \lambda_j\langle P_j, P_{\le d}-\hat{P}_{\le d}\rangle.
\end{equation*}
By $P_{\le d}-\hat{P}_{\le d}=\hat{P}_{> d}-P_{> d}$ we obtain
\begin{align*}
{\mathcal E}^{PCA}_d&=\sum_{j\le d}\lambda_j\langle P_j,\hat{P}_{> d}-P_{>d}\rangle-\sum_{k> d} \lambda_{k}\langle P_k,\hat{P}_{\le d}-P_{\le d}\rangle\nonumber\\
&=\sum_{j\le d}\lambda_j\langle P_j,\hat{P}_{> d}\rangle-\sum_{k>d} \lambda_{k}\langle P_k,\hat{P}_{\le d}\rangle.
\end{align*}
Moreover, the identity
\[
\langle P_{\le d},\hat{P}_{> d}\rangle =\langle P_{\le d},\hat{P}_{> d}-P_{>d}\rangle=-\langle P_{> d},P_{\le d}-\hat{P}_{\le d}\rangle=\langle P_{> d},\hat{P}_{\le d}\rangle
\]
implies $\sum_{j\le d}\mu\scapro{P_j}{\hat P_{> d}}=\sum_{k>d} \mu\scapro{P_k}{\hat P_{\le d}}$ and thus
\begin{align*}
{\mathcal E}^{PCA}_d=\sum_{j\le d}(\lambda_j-\mu)\langle P_j,\hat{P}_{> d}\rangle
+\sum_{k>d}(\mu-\lambda_{k})\langle P_k,\hat{P}_{\le d}\rangle.
\end{align*}
The claim now follows from inserting the identities $\langle P_j,\hat{P}_{>d}\rangle=\Vert P_j\hat{P}_{>d}\Vert_2^2$ and $\langle P_k,\hat{P}_{\le d}\rangle=\Vert P_k\hat{P}_{\le d}\Vert_2^2$.\qed

\subsection{Proof of Proposition \ref{PropMainRes1}}
Taking expectation in \eqref{EqErrorDecR1} and using $\mathbb{E}[ \Vert P_{j}\Delta \hat{P}_{> d} \Vert_2^2] \le \mathbb{E}[ \Vert P_{j}\Delta \Vert_2^2] \le 2C_2\lambda_j\operatorname{tr}(\Sigma)/n$, we get
\begin{align*}
\mathbb{E}\left[{\mathcal E}^{PCA}_{\le d}(\mu)\right]
&\le 8C_2\sum_{j\le r}(\lambda_j-\mu)\frac{\lambda_j\tr (\Sigma) }{n(\lambda_j-\lambda_{d+1})^2}+\sum_{j=r+1}^{d\wedge (r+p-d)}(\lambda_j-\mu)\\
&+ \sum_{j\le r}(\lambda_j-\mu) \mathbb{P}( \hat{\lambda}_{d+1}-\lambda_{d+1} > (\lambda_j-\lambda_{d+1})/2).
\end{align*}
Hence, the first inequality follows from the following lemma.

\begin{lemma}\label{Lem1Thm1} Let $j\le d$. Then
\begin{align*}
&8C_2\frac{\lambda_j\tr (\Sigma) }{n(\lambda_j-\lambda_{d+1})^2}+\mathbb{P}( \hat{\lambda}_{d+1}-\lambda_{d+1} > (\lambda_j-\lambda_{d+1})/2)\\
&\le (8C_2+4C_3^2)\frac{\lambda_j\tr (\Sigma) }{n(\lambda_j-\lambda_{d+1})^2}.
\end{align*}
\end{lemma}

\begin{proof}[Proof of Lemma \ref{Lem1Thm1}]
If
\begin{equation}\label{eq:dscases}
\frac{\lambda_j\tr (\Sigma) }{n(\lambda_j-\lambda_{d+1})^2}\le 1/(4C_3^3),
\end{equation}
then Condition \eqref{CondCorRightDev} is satisfied and we can apply \eqref{EqCorRightDev}. Thus, in this case, the left-hand side can be bounded by
\[
8C_2\frac{\lambda_j\tr (\Sigma) }{n(\lambda_j-\lambda_{d+1})^2}+\exp\bigg( -\frac{n(\lambda_j-\lambda_{d+1})^2}{4C_3^2\lambda_j^2}\bigg)\le (8C_2+2C_3^2)\frac{\lambda_j\tr (\Sigma) }{n(\lambda_j-\lambda_{d+1})^2},
\]
where the inequality follows from $x\exp(-x)\le 1/e\le 1/2$, $x\ge 0$. On the other hand, if \eqref{eq:dscases} is not satisfied, then the left-hand side can be bounded by
\[
8C_2\frac{\lambda_j\tr (\Sigma) }{n(\lambda_j-\lambda_{d+1})^2}+1\le (8C_2+4C_3^2)\frac{\lambda_j\tr (\Sigma) }{n(\lambda_j-\lambda_{d+1})^2}.
\]
Hence, we get the claim in both cases.
\end{proof}
It remains to prove the second inequality. Taking expectation in \eqref{EqErrorDecL1} and using $\mathbb{E}[ \Vert P_{k}\Delta \hat{P}_{\le d}\Vert_2^2] \le \mathbb{E}[ \Vert P_{k}\Delta \Vert_2^2] \le 2C_2\lambda_k\tr (\Sigma)/n $, we get
\begin{align*}
\mathbb{E}\left[ {\mathcal E}^{PCA}_{> d}(\mu)\right]
&\le 8C_2\sum_{k\ge  l}(\mu-\lambda_k)\frac{\lambda_k\tr (\Sigma) }{n(\lambda_d-\lambda_k)^2}+\sum_{k=(d+1)\vee(l-d)}^{l-1}(\mu-\lambda_k)\\
&+\sum_{\substack{k\ge  l:\\ \lambda_k\ge \lambda_d/2}}(\mu-\lambda_k) \mathbb{P}( \hat{\lambda}_{d}-\lambda_{d}<-(\lambda_d-\lambda_k)/2)\\
&+d(\mu-\lambda_p) \mathbb{P}(\hat{\lambda}_d-\lambda_d< -\lambda_d/4).
\end{align*}
Consider the last term. Since $d\le n/(16C_3^2)$, Condition \eqref{EqLeftDev} is satisfied with $x=\lambda_d/4$  and we obtain
\begin{equation*}
\mathbb{P}( \hat{\lambda}_d-\lambda_d< -\lambda_d/4) \le  e^{ -\frac{n}{16C_3^2}}.
\end{equation*}
Using this and the bounds $x\exp(-x)\le (2/e)\exp(-x/2)\le \exp(-x/2)$, $x\ge 0$, and $d\le n/(16C_3^2)$, we see that
\begin{equation*}
d \mathbb{P}( \hat{\lambda}_d-\lambda_d< -\lambda_d/4)\le de^{ -\frac{n}{16C_3^2}}\le e^{ -\frac{n}{32C_3^2}}.
\end{equation*}
This gives the desired bound for the last term. Now, the second inequality follows from the lemma below.
\begin{lemma}\label{Lem1Thm2} Let $k>d$ be such that $\lambda_k\ge \lambda_d/2$. Then
\begin{align*}
&8C_2\frac{\lambda_k\tr (\Sigma) }{n(\lambda_d-\lambda_k)^2}+\mathbb{P}( \hat{\lambda}_{d}-\lambda_{d}<-(\lambda_d-\lambda_k)/2)\\
&\le (8C_2+8C_3^2)\frac{\lambda_k\tr (\Sigma) }{n(\lambda_d-\lambda_k)^2}.
\end{align*}
\end{lemma}

\begin{proof}[Proof of Lemma \ref{Lem1Thm2}]
If
\begin{equation}\label{eq:ikjcases}
\frac{\lambda_d\tr (\Sigma) }{n(\lambda_d-\lambda_k)^2}\le 1/(4C_3^2),
\end{equation}
then Condition \eqref{CondCorLeftDev} is satisfied and we can apply \eqref{EqCorLeftDev}. Thus, in this case, the left-hand side can be bounded by
\[
8C_2\frac{\lambda_k\tr (\Sigma) }{n(\lambda_d-\lambda_k)^2}+\exp\bigg( -\frac{n(\lambda_d-\lambda_k)^2}{4C_3^2\lambda_d^2}\bigg).
\]
Using the inequality $x\exp(-x)\le 1/e\le 1/2$, $x\ge 0$, this is bounded by
\[
8C_2\frac{\lambda_k\tr (\Sigma) }{n(\lambda_d-\lambda_k)^2}+\frac{2C_3^2\lambda_d^2}{n(\lambda_d-\lambda_k)^2}\le (8C_2+4C_3^2)\frac{\lambda_k\tr (\Sigma) }{n(\lambda_d-\lambda_k)^2},
\]
where we applied the bound $\lambda_d/2\le \lambda_k$ in the second inequality.
On the other hand, if \eqref{eq:ikjcases} is not satisfied, then the left-hand side can be bounded by
\[
8C_2\frac{\lambda_k\tr (\Sigma) }{n(\lambda_d-\lambda_k)^2}+1\le (8C_2+8C_3^2)\frac{\lambda_k\tr (\Sigma) }{n(\lambda_d-\lambda_k)^2},
\]
where we again applied the bound $\lambda_d/2\le \lambda_k$. Hence, we get the claim in both cases.
\end{proof}

\subsection{Proof of Corollary \ref{MainRes1Var}}
The claim follows from Proposition \ref{PropMainRes1} together with the facts that for $\mu\in [\lambda_{d+1},\lambda_d]$ the terms $\lambda_j-\mu$ (resp. $\mu-\lambda_k$) can be upper bounded by $\lambda_j-\lambda_{d+1}$ (resp. $\lambda_d-\lambda_k$) and that $\lambda\mapsto \lambda/(\lambda-\lambda_{d+1})^2$ is decreasing for $\lambda>\lambda_{d+1}$ (resp. $\lambda\mapsto \lambda/(\lambda_d-\lambda)^2$ is increasing for $\lambda<\lambda_{d}$).\qed

\subsection{Proof of Theorem \ref{ThmGlobalLocalBound}}
In Corollary \ref{MainRes1Var}, only summands with $\lambda_j>\lambda_{d+1}$ and $\lambda_k<\lambda_d$, respectively, appear. Neglecting the minimum with $\lambda_j-\lambda_{d+1}$ (resp. $\lambda_d-\lambda_k$) in each summand, the local bound follows.

For the global bound use the inequality $\min(a/x,x)\le\sqrt{a}$ for $a,x\ge 0$ to obtain from Corollary \ref{MainRes1Var}
\begin{equation*}
\mathbb{E}[{\mathcal E}^{PCA}_{\le d}(\mu)]\le \sum_{j\le d}\sqrt{\frac{C\lambda_j\operatorname{tr}(\Sigma)}{n}}.
\end{equation*}
Considering ${\mathcal E}^{PCA}_{>d}(\mu)$, the value $l$ in Proposition \ref{PropMainRes1} has to be chosen carefully.  For $a>0$ let $d< l=l(a)\le p+1$ be the index such that $\lambda_d-\lambda_k\ge a$ for $k\ge l$ and $\lambda_d-\lambda_k< a$ for $d<k< l$. Then the second inequality of Proposition \ref{PropMainRes1} and the inequality $\mu\le\lambda_{d}$ imply 
\begin{equation*}
\mathbb{E}[{\mathcal E}^{PCA}_{> d}(\mu)]\le \sum_{k>d}\frac{C\lambda_k\operatorname{tr}(\Sigma)}{na}+da+\lambda_de^{ -\frac{n}{32C_3^2}}.
\end{equation*}
Minimizing over $a>0$ and incorporating the remainder in the summand for $j=d$ gives the global bound in \eqref{EqGlBound}.
\qed

\subsection{Proof of Proposition \ref{MainRes2}}\label{ProofThm3}
We begin with the following extension of Lemma \ref{Lem1Thm1}.  
\begin{lemma}\label{lemma1mthm3}
Let $j\le s\le d$. Then
\begin{align}
&16C_2\frac{\lambda_j\operatorname{tr}_{>s}(\Sigma)}{n(\lambda_j-\lambda_{d+1})^2}+ 2\mathbb{P}( \hat{\lambda}_{d+1}-\lambda_{d+1}> (\lambda_j-\lambda_{d+1})/2)\label{eq:proofmthm31}\\
& \le (16C_2+8C_3^2)\frac{\lambda_j\operatorname{tr}_{>s}(\Sigma)}{n(\lambda_j-\lambda_{d+1})^2}\nonumber.
\end{align}
\end{lemma}
\begin{proof}[Proof of Lemma \ref{lemma1mthm3}]
If
\begin{equation}\label{eq:cond1}
\frac{\lambda_j\operatorname{tr}_{>s}(\Sigma)}{n(\lambda_j-\lambda_{d+1})^2}\le 1/(4C_3^2),
\end{equation}
then the Condition \eqref{CondCorRightDev} is satisfied and the left-hand side of \eqref{eq:proofmthm31} can be bounded by
\begin{align*}
&16C_2\frac{\lambda_j\operatorname{tr}_{>s}(\Sigma)}{n(\lambda_j-\lambda_{d+1})^2}+2\exp\bigg( -n\min\bigg( \frac{(\lambda_j-\lambda_{d+1})^2}{4C_3^2\lambda_{d+1}^2},\frac{\lambda_j-\lambda_{d+1}}{2C_3\lambda_{d+1}}\bigg) \bigg)\\
&\le (16C_2+8C_3^2)\frac{\lambda_j\operatorname{tr}_{>s}(\Sigma)}{n(\lambda_j-\lambda_{d+1})^2},
\end{align*}
where we used the bounds $x\exp(-x)\le 1/e\le 1/2$ and $x^2\exp(-x)\le 4/e^2\le 1$.
On the other hand, if \eqref{eq:cond1} is not satisfied, then
\eqref{eq:proofmthm31} can be bounded by
\[
16C_2\frac{\lambda_j\operatorname{tr}_{>s}(\Sigma)}{n(\lambda_j-\lambda_{d+1})^2}+2\le (16C_2+8C_3^2)\frac{\lambda_j\operatorname{tr}_{>s}(\Sigma)}{n(\lambda_j-\lambda_{d+1})^2}.
\]
This completes the proof.
\end{proof}
Taking expectation in \eqref{EqErrorDecR2} with $\mu=\lambda_{d+1}$ and using Lemma \ref{lemma1mthm3}, we obtain
\begin{align*}
\mathbb{E}\left[ {\mathcal E}^{PCA}_{\le d}(\lambda_{d+1})\right]
&\le  (16C_2+8C_3^2)\sum_{j\le r}\frac{\lambda_j\tr_{>s}(\Sigma)}{n(\lambda_j-\lambda_{d+1})} +2\sum_{r<j\le d}(\lambda_j-\lambda_{d+1})\\
&+8\mathbb{E}\Bigg[ \sum_{j\le r}\frac{\Vert P_{j}\Delta \Vert_2^2}{\lambda_j-\lambda_{d+1}} \mathbbm{1}\big(\| S_{\le s}\Delta S_{\le s}\|_\infty > 1/4\big)\Bigg].
\end{align*}
If Condition \eqref{EqThm3Cond} holds, then Lemma \ref{LemWeightedCov} with $\mu=\lambda_{d+1}$ gives
\begin{equation}\label{eq:kl1}
\mathbb{P}\left( \|S_{\le s}\Delta S_{\le s}\|_\infty > 1/4\right) \le \exp\bigg( -\frac{n(\lambda_{s}-\lambda_{d+1})^2}{16C_3^2\lambda^2_{s}}\bigg).
\end{equation}
Thus the claim follows from applying the Cauchy-Schwarz inequality, \eqref{eq:kl1}, and the following  lemma:
\begin{lemma}\label{lemma2mthm3}
For all $r\le d$, we have
\[
\Bigg( \mathbb{E}\Bigg[ \Bigg( \sum_{j\le r}\frac{\Vert P_{j}\Delta \Vert_2^2}{\lambda_j-\lambda_{d+1}}\Bigg)^2\Bigg]\Bigg) ^{1/2} \le  128C_1\sum_{j\le r}\frac{\lambda_j\tr (\Sigma) }{n(\lambda_j-\lambda_{d+1})}.
\]
\end{lemma}
\begin{proof}[Proof of Lemma \ref{lemma2mthm3}]
By the Minkovski inequality, we have
\begin{equation}\label{eq:eq1}
\Bigg( \mathbb{E}\Bigg[ \Bigg( \sum_{j\le r}\frac{\Vert P_{j}\Delta \Vert_2^2}{\lambda_j-\lambda_{d+1}}\Bigg)^2\Bigg]\Bigg) ^{1/2}\le \sum_{j\le r}\sum_{k\ge 1}\frac{\left( \mathbb{E}\left[ \Vert P_{j}\Delta P_k\Vert_2^4 \right]\right) ^{1/2}}{\lambda_j-\lambda_{d+1}}.
\end{equation}
Using Assumption \ref{SubGauss}, it is simple to check that
\begin{equation}\label{eq:eq2}
\left( \mathbb{E}\left[ \Vert P_{j}\Delta P_k\Vert_2^4 \right]\right) ^{1/2}\le  128C_1\lambda_j\lambda_k/n,
\end{equation}
and the claim follows from inserting this into \eqref{eq:eq1}.
\end{proof}

\subsection{Proof of Theorem \ref{ThmLocalGlobalImpr}}
By assumption, we have $\lambda_j-\lambda_p\le c_1^{-1}(\lambda_j-\lambda_{d+1})$ for all $j\le d$. Thus, Lemma \ref{SpectralDec} applied with $\mu=\lambda_p$ yields
\begin{equation}\label{EqERRed1}
{\mathcal E}^{PCA}_d\le \sum_{j\le d}(\lambda_j-\lambda_p)\Vert P_{j}\hat{P}_{> d}\Vert_2^2\le c_1^{-1}{\mathcal E}^{PCA}_{\le d}(\lambda_{d+1}).
\end{equation}
The local bound now follows from Proposition \ref{MainRes2} applied with $r=s=d$. 

It remains to prove the global bound. We begin with the following lemma.

\begin{lemma}\label{hilfslem} Let $s\le d$ be the largest number such that Condition \eqref{EqThm3Cond} is satisfied (and $s=0$ if such a number does not exist). Then we have
\begin{equation*}
\sum_{s<j\le d}(\lambda_j-\lambda_{d+1}) \le 16C_3^2\sum_{j\le s}\frac{\lambda_j \tr_{>s}(\Sigma)}{n(\lambda_j-\lambda_{d+1})}+\sum_{s<j\le d}\frac{4\sqrt{2}C_3\lambda_j}{\sqrt{n}}.
\end{equation*}
\end{lemma}

\begin{proof}
If Condition \eqref{EqThm3Cond} does not hold for $s+1$, then we have
\begin{equation*}
\frac{\lambda_{s+1}-\lambda_{d+1}}{\lambda_{s+1}}<16C_3^2\sum_{j\le s+1}\frac{\lambda_j}{n(\lambda_j-\lambda_{d+1})}.
\end{equation*}
Multiplying both sides by $(\lambda_{s+1}-\lambda_{d+1})/\lambda_{s+1}$, we obtain
\[
\left( \frac{\lambda_{s+1}-\lambda_{d+1}}{\lambda_{s+1}}\right) ^2<\left( \frac{\lambda_{s+1}-\lambda_{d+1}}{\lambda_{s+1}}\right)16C_3^2\sum_{j\le s}\frac{\lambda_j}{n(\lambda_j-\lambda_{d+1})}+\frac{16C_3^2}{n}.
\]
Using that for $x,a,b\ge 0$ the inequality $x^2\le ax+b$ implies $x^2\le a^2+2b$, which in turn implies $x\le a+\sqrt{2b}$, we obtain 
\[
\frac{\lambda_{k}-\lambda_{d+1}}{\lambda_{k}}\le \frac{\lambda_{s+1}-\lambda_{d+1}}{\lambda_{s+1}}<16C_3^2\sum_{j\le s}\frac{\lambda_j}{n(\lambda_j-\lambda_{d+1})}+\frac{4\sqrt{2}C_3}{\sqrt{n}}
\]
for all $s<k\le d$. Multiplying both sides with $\lambda_{k}$ and summing over $s<k\le d$, the claim follows.
\end{proof}
Now let us finish the proof of the global bound.
Set
\[
A=\tr_{> s}(\Sigma)+\tr(\Sigma)\exp\bigg( -\frac{nc_1^2(\lambda_{s}-\lambda_{p})^2}{16C_3^2\lambda^2_{s}}\bigg)
\]
(and $A=\tr(\Sigma)$ if $s=0$) and let $j_0\le d$ be the unique number such that
\begin{equation}\label{EqDefj_0}
\frac{\lambda_{j}A}{n(\lambda_{j}-\lambda_{d+1})^2}\le 1\Leftrightarrow j\le j_0
\end{equation}
(and $j_0=0$ if such a number does not exist)
If $s\le j_0$, then Proposition \ref{MainRes2} with $r=s$ (resp. Corollary \ref{MainRes1Var} if $s=0$) and Lemma \ref{hilfslem} give
\begin{align*}
\mathbb{E}\left[{\mathcal E}^{PCA}_{\le d}(\lambda_{d+1})\right]&\le C\sum_{j\le s}\frac{\lambda_jA}{n(\lambda_j-\lambda_{d+1})}+2\sum_{s<j\le d}(\lambda_j-\lambda_{d+1})\nonumber\\
&\le (C+32C_3^2)\sum_{j\le s}\frac{\lambda_jA}{n(\lambda_j-\lambda_{d+1})}+\sum_{s<j\le d}\frac{8\sqrt{2}C_3\lambda_j}{\sqrt{n}}.
\end{align*}
From $s\le j_0$ we infer
\[
\frac{\lambda_{j}A}{n(\lambda_{j}-\lambda_{d+1})^2}\le 1
\]
for all $j\le s$. Using this and for $a,x\ge 0$ the implication $a/x^2\le 1 \Rightarrow a/x\le \sqrt{a}$, we find
\begin{equation}\label{Eqbound1<=d}
\mathbb{E}\left[{\mathcal E}^{PCA}_{\le d}(\lambda_{d+1})\right]\le (C+32C_3^2)\sum_{j\le s}\sqrt{\frac{\lambda_jA}{n}}+\sum_{s<j\le d}\frac{8\sqrt{2}C_3\lambda_j}{\sqrt{n}}.
\end{equation}
On the other hand, if $s>j_0$, then applying Proposition \ref{MainRes2} with $r=j_0$ yields
\begin{align}
\mathbb{E}\left[{\mathcal E}^{PCA}_{\le d}(\lambda_{d+1})\right]
&\le C\sum_{j\le j_0}\frac{\lambda_jA}{n(\lambda_j-\lambda_{d+1})}+2\sum_{j_0<j\le d}(\lambda_j-\lambda_{d+1})\label{Eqbound2<=d}\\
&\le C\sum_{j\le d}\sqrt{\frac{\lambda_jA}{n}}\nonumber
\end{align}
with $C>2$. Plugging the definition of $A$ into \eqref{Eqbound1<=d} and \eqref{Eqbound2<=d}, the claim follows from \eqref{EqERRed1} and the inequality $\sqrt{x+y}\le\sqrt{x}+\sqrt{y}$, $x,y\ge 0$.\qed

\subsection{Proof of Theorem \ref{ThmRelative}} Similarly as in \eqref{EqERRed1}, we have
\begin{equation*}
{\mathcal E}^{PCA}_d\le \sum_{j\le d}\lambda_j\Vert P_{j}\hat{P}_{> d}\Vert_2^2\le \frac{\lambda_s}{\lambda_s-\lambda_{d+1}}\sum_{j\le s}(\lambda_j-\lambda_{d+1})\Vert P_{j}\hat{P}_{> d}\Vert_2^2+\tr_{> s}(\Sigma).
\end{equation*}
By \eqref{EqExpBound} and \eqref{EqRecBound2} with $\mu=\lambda_{d+1}$ and $r=s$, we have
\begin{align*}
&\sum_{j\le s}(\lambda_j-\lambda_{d+1})\Vert P_{j}\hat{P}_{> d}\Vert_2^2
\\
&\le 16\sum_{j\le s}\frac{\Vert P_{j}\Delta P_{>s}\Vert_2^2}{\lambda_j-\lambda_{d+1}}
+ 2\sum_{j\le s}(\lambda_j-\lambda_{d+1}) \mathbbm{1}\big(\hat{\lambda}_{d+1}-\lambda_{d+1}> (\lambda_j-\lambda_{d+1})/2\big)\nonumber\\
&+8\sum_{j\le s}\frac{\Vert P_{j}\Delta\Vert_2^2}{\lambda_j-\lambda_{d+1}} \mathbbm{1}\big(\|S_{\le s}\Delta S_{\le s}\|_\infty > 1/4\big).
\end{align*}
As shown in the proof of Proposition \ref{MainRes2} the inequality
\[
\mathbb{E}\Big[\sum_{j\le s}(\lambda_j-\lambda_{d+1})\Vert P_{j}\hat{P}_{> d}\Vert_2^2\Big]\le C\sum_{j\le s}\frac{\lambda_j\tr_{> s}(\Sigma)}{n(\lambda_j-\lambda_{d+1})}+R
\]
holds with remainder term $R$ given in Proposition \ref{MainRes2} with $r=s$. Inserting this into the above inequality and using Condition \eqref{EqThm3Cond}, we get
\begin{align*}
\mathbb{E}[{\mathcal E}^{PCA}_d]\le C\tr_{> s}(\Sigma)+C\operatorname{tr}(\Sigma)\exp\bigg( -\frac{n(\lambda_{s}-\lambda_{d+1})^2}{16C_3^2\lambda^2_{s}}\bigg)
\end{align*}
with a constant $C>0$ depending only on $C_1$.\qed

\appendix

\section{Proof of Proposition \ref{PropAsympt}}\label{ProofPropAsympt}
We begin by recalling the following asymptotic result from multivariate analysis. By  \cite[Proposition 5]{DPR82}, we have
\begin{equation}\label{EqCLT}
\sqrt{n}\Delta=\sqrt{n}\Delta_n\xrightarrow{d}L=\sum_{j= 1}^p\sum_{k= 1}^p\sqrt{\lambda_j\lambda_k}\xi_{jk} (u_j\otimes u_k),
\end{equation}
where the upper triangular coefficients $\xi_{jk}$, $k\ge j$ are independent and centered Gaussian random variables, with variance $1$ for $k>j$ and variance $2$ for $k=j$. The lower triangular coefficients $\xi_{jk}$, $k< j$ are determined by $\xi_{jk}=\xi_{kj}$.
For $l= 1,\dots,p$, let $I_l=\{j:\lambda_j=\lambda_l\}$,
\[
Q_l=\sum_{j\in I_l}P_j,
\text{ and }
\hat{Q}_l=\sum_{j\in I_l}\hat{P}_j.
\]
For $l\ge 1$, we have
\[
Q_lLQ_l=\lambda_l\sum_{j,k\in I_l}\xi_{jk} (u_j\otimes u_k)
\]
and the random matrix
\[
M_l=(\xi_{jk})_{j,k\in I_l}
\]
is a GOE matrix. It is well known (see e.g. \cite{T12}) that the eigenvalues of $M_l$ (in decreasing order) have a joint density with respect to the Lebesgue measure on $\mathbb{R}^{m_l}_{\ge}$, where $m_l=|I_l|$, and that the matrix of the corresponding eigenvectors is distributed according to the Haar measure on the orthogonal group $O(m_l)$. It is easy to see that $M_l$ and $Q_lLQ_l$ have the same non-zero eigenvalues and that a matrix of eigenvectors of $M_l$ gives the coefficients of a basis of eigenvectors of $Q_lLQ_l$ with respect to the basis $(u_j)_{j\in I_l}$. In particular, with probability $1$, $Q_lLQ_l$ has rank $m_l$ and all $m_l$ non-zero eigenvalues are distinct. Let $(P^{Haar}_j)_{j\in I_l}$ denote the corresponding spectral projectors (ordered such that the orthogonal projection onto the eigenvector corresponding to the largest non-zero eigenvalue appears first, and so forth). The following key observation is proved below.
 \begin{lemma}\label{HaarMeasConv}
We have
\begin{equation*}
(\sqrt{n}\Delta,\hat{P}_1,\dots,\hat{P}_p)\xrightarrow{d}(L,P_1^{Haar},\dots,P_p^{Haar}).
\end{equation*}
\end{lemma}
Lemma \ref{HaarMeasConv} implies Proposition \ref{PropAsympt}. By  Lemma \ref{SpectralDec} and Lemma \ref{FirstOrdExp}, we have
\begin{align}
{\mathcal E}^{PCA}_{d,n}=\sum_{\substack{j\le d,k>d:\\\lambda_j>\lambda_d}} (\lambda_j-\lambda_d)\frac{\|P_j\Delta\hat{P}_k\|_2^2}{(\lambda_j-\hat{\lambda}_k)^2}+\sum_{\substack{j\le d,k>d:\\\lambda_k<\lambda_d}}(\lambda_d-\lambda_k)\frac{\|P_k\Delta\hat{P}_j\|_2^2}{(\hat{\lambda}_j-\lambda_k)^2}.\label{EqKeyRepLimThm}
\end{align}
Using Lemma \ref{HaarMeasConv}, \eqref{EqKeyRepLimThm}, the fact that $\hat{\lambda}_l\xrightarrow{a.s.}\lambda_l$ for all $l$ (see \cite[Proposition 2]{DPR82}), and the continuous mapping theorem, we thus conclude that
\begin{align*}
n{\mathcal E}^{PCA}_{d,n}
&\xrightarrow{d}\sum_{\substack{j\le d,k>d:\\\lambda_j>\lambda_d,\lambda_k<\lambda_d}} \frac{\|P_jLP_k\|_2^2}{\lambda_j-\lambda_k}\\
&+\sum_{\substack{j\le d,k>d:\\\lambda_j>\lambda_d, \lambda_k=\lambda_d}}(\lambda_j-\lambda_d)\frac{\|P_jL P_k^{Haar}\|_2^2}{(\lambda_j-\lambda_k)^2}\\
&+\sum_{\substack{j\le d,k>d:\\\lambda_j=\lambda_d,\lambda_k<\lambda_d}}(\lambda_d-\lambda_k)\frac{\|P_kL P_j^{Haar}\|_2^2}{(\lambda_j-\lambda_k)^2},
\end{align*}
where we also used the identities $\sum_{j\in I_l}P_j^{Haar}=Q_l=\sum_{j\in I_l}P_j$ in the first summand of the limit.
Further, note that the tuple $(P_j^{Haar})_{j\in I_d}$ is independent of $\{P_j L P_k: \lambda_j\neq \lambda_d\text{ or }\lambda_k\neq \lambda_d\}$. Hence, the random variables $\|P_jLP_k\|_2^2/(\lambda_j\lambda_k)$, $\|P_jL P_k^{Haar}\|_2^2/(\lambda_j\lambda_k)$, and $\|P_kL P_j^{Haar}\|_2^2/(\lambda_j\lambda_k)$ appearing in the above sums are independent and chi-square distributed. This gives Proposition \ref{PropAsympt}. It remains to prove Lemma \ref{HaarMeasConv}.
\begin{proof}[Proof of Lemma \ref{HaarMeasConv}]
Let $I\subseteq \{1,\dots,p\}$ be a subset which contains for each $l\ge 1$ exactly one element of $I_l$. Then we have
\[
(\sqrt{n}\Delta,(Q_l\sqrt{n}\Delta Q_l)_{l\in I})\xrightarrow{d}(L,(Q_l L Q_l)_{l\in I}).
\]
In \cite[Section 2.2.1]{DPR82}, it is shown that for each $l\in I$,
\begin{equation}
\hat{Q}_l \sqrt{n}(\hat{\Sigma}-\lambda_l I)\hat{Q}_l-Q_l \sqrt{n}\Delta Q_l\xrightarrow{\mathbb{P}}0.\label{eq:slutskyeigensp}
\end{equation}
By Slutsky's lemma, we thus have
\begin{equation}\label{eq:wGOE}
(\sqrt{n}\Delta,(\hat{Q}_l \sqrt{n}(\hat{\Sigma}-\lambda_l I)\hat{Q}_l)_{l\in I})\xrightarrow{d} (L,(Q_l LQ_l)_{l\in I}).
\end{equation}
We now apply the continuous mapping theorem. For $l\in I$, let $h_l$ be a mapping sending a symmetric $p\times p$ matrix of rank $m_l$ to the spectral projectors corresponding to the $m_l$ non-zero eigenvalues (ordered such that the orthogonal projection onto the eigenvector corresponding to the largest non-zero eigenvalue appears first, and so forth). Note that this map is uniquely determined and continuous if restricted to the open subset of symmetric matrices of rank $m_l$ having distinct non-zero eigenvalues. As already argued above, with probability $1$, $Q_lLQ_l$ has rank $m_l$ and all $m_l$ non-zero eigenvalues are distinct. Moreover, since $X$ is Gaussian, the same is true for $\hat{Q}_l \sqrt{n}(\hat{\Sigma}-\lambda_l I)\hat{Q}_l$. Thus, the claim follows from \cite[Theorem 2.7]{B99} applied to $h=(I,(h_l)_{l\in I})$.
\end{proof}

\section{Linear expansions}\label{SecLinExp}
In this complementary section, we show that linear expansions of $\hat{P}_{> d}$ and $\hat{P}_{\le d}$ may lead to tight bounds for the excess risk as well as for the Hilbert-Schmidt distance if stronger eigenvalue conditions are satisfied (including $\lambda_d>\lambda_{d+1}$). In particular, these bounds lead to the exact leading terms appearing in the asymptotic limit results in \eqref{EqAsympSinTheta} and Proposition \ref{PropAsympt}. Note that it is possible to derive higher order expansions by similar, but more tedious considerations.

\subsection{Main deterministic upper bounds} Proceeding as in Section \ref{SecErrorDec}, we have the following upper bound for the Hilbert-Schmidt distance.
\begin{proposition}\label{LinExpSinTheta}
On the event $\{\hat{\lambda}_{d+1}-\lambda_{d+1}\le (\lambda_d-\lambda_{d+1})/2\}$ we have
\begin{equation}\label{EqLinExpSinTheta}
\bigg\|P_{\le d}\hat P_{>d}-\sum_{j\le d}\sum_{k>d} \frac{P_j\Delta P_k}{\lambda_j-\lambda_k}\bigg\|_2^2\le \bigg\|\sum_{j\le d}\sum_{k>d} \frac{P_j\Delta P_k}{\lambda_j-\lambda_k}\bigg\|_2^2\|P_{> d}\hat P_{\le d}\|_\infty^2+R_1
\end{equation}
with remainder term $R_1$ given by
\begin{align*}
R_1&=16\sum_{j\le d}\frac{1}{(\lambda_j-\lambda_{d+1})^2}\Bigg\| \sum_{k>d}\frac{P_j\Delta P_k\Delta P_{>d}}{\lambda_j-\lambda_k}\Bigg\|_2^2+8\|S_{\le d}^2\Delta S_{\le d}\|_\infty^2 {\mathcal E}^{PCA}_{\le d}(\lambda_{d+1})\\
&+16\sum_{j\le d}\frac{1}{(\lambda_j-\lambda_{d+1})^2}\Bigg\| \sum_{k>d}\frac{P_j\Delta P_k\Delta S_{\le d}}{\lambda_j-\lambda_k}\Bigg\|_2^2 {\mathcal E}^{PCA}_{\le d}(\lambda_{d+1})
\end{align*}
with $S_{\le d}=S_{\le d}(\lambda_{d+1})$ from \eqref{EqWeightOper}.
In particular, letting
\[
L_{1}(\Delta)=\sum_{j\le d}\sum_{k>d} \frac{P_j\Delta P_k}{\lambda_j-\lambda_k},
\]
we have, on the event $\{\hat{\lambda}_{d+1}-\lambda_{d+1}\le (\lambda_d-\lambda_{d+1})/2\}$,
\begin{equation}\label{EqLinExpSinThetaC1}
\|\hat{P}_{\le d}-P_{\le d}\|_2^2\le 8\|L_{1}(\Delta)\|_2^2+4R_1
\end{equation}
and
\begin{align}\label{EqLinExpSinThetaC2}
&|\|\hat{P}_{\le d}-P_{\le d}\|_2^2-2\|L_{1}(\Delta)\|_2^2|\\
&\le C(\|L_{1}(\Delta)\|_2^4+\|L_{1}(\Delta)\|_2^3+R_1\|L_{1}(\Delta)\|_2^2+\sqrt{R_1}\|L_{1}(\Delta)\|_2+R_1)\nonumber
\end{align}
with some absolute constant $C>0$.
\end{proposition}

\begin{proof}
In the proof of Lemma \ref{LemLinExp} we have also shown that for $j\le d$,
\begin{align}
&P_j\hat{P}_{>d}-\sum_{k>d}\frac{P_j\Delta P_k}{\lambda_j-\lambda_k}\label{eq:soev}\\
&=\sum_{k>d}\frac{P_j\Delta P_k\hat{P}_{\le d}}{\lambda_j-\lambda_k}+\sum_{l>d}\frac{P_j\Delta P_{\le d}\hat{P}_l}{\lambda_j-\hat{\lambda}_l}-\sum_{k>d}\sum_{l>d}\frac{P_j\Delta P_k\Delta \hat{P}_l}{(\lambda_j-\hat{\lambda}_l)(\lambda_j-\lambda_k)}.\nonumber
\end{align}
Summing over $j\le d$ and taking the Hilbert-Schmidt norm, we obtain
\begin{align*}
&\Bigg\|P_{\le d}\hat{P}_{>d}-\sum_{j\le d}\sum_{k>d}\frac{P_j\Delta P_k}{\lambda_j-\lambda_k}\Bigg\|_2^2\\
&\le \Bigg\| \sum_{j\le d}\sum_{k>d}\frac{P_j\Delta P_k\hat{P}_{\le d}}{\lambda_j-\lambda_k}\Bigg\|_2^2+2\Bigg\| \sum_{j\le d}\sum_{l>d}\frac{P_j\Delta P_{\le d}\hat{P}_{l}}{\lambda_j-\hat{\lambda}_l}\Bigg\|_2^2\\
&+2\Bigg\| \sum_{j\le d}\sum_{k>d}\sum_{l>d}\frac{P_j\Delta P_k\Delta \hat{P}_l}{(\lambda_j-\hat{\lambda}_l)(\lambda_j-\lambda_k)}\Bigg\|_2^2,
\end{align*}
where we also applied the triangle inequality. Thus, on the event $\{\hat{\lambda}_{d+1}-\lambda_{d+1}\le (\lambda_d-\lambda_{d+1})/2\}$ we have
\begin{align}
&\Bigg\|P_{\le d}\hat{P}_{>d}-\sum_{j\le d}\sum_{k>d}\frac{P_j\Delta P_k}{\lambda_j-\lambda_k}\Bigg\|_2^2\label{EqSi}\\
&\le \Bigg\|\sum_{j\le d}\sum_{k>d} \frac{P_j\Delta P_k}{\lambda_j-\lambda_k}\Bigg\|_2^2\nonumber\|P_{>d}\hat{P}_{\le d}\|_\infty^2+8\sum_{j\le d}\sum_{l>d}\frac{\|P_j\Delta P_{\le d}\hat{P}_{l}\|_2^2}{(\lambda_j-\lambda_{d+1})^2}\nonumber\\
&+8\sum_{j\le d}\sum_{l>d}\frac{1}{(\lambda_j-\lambda_{d+1})^2}\Bigg\| \sum_{k>d}\frac{P_j\Delta P_k\Delta \hat{P}_{l}}{\lambda_j-\lambda_k}\Bigg\|_2^2\nonumber.
\end{align}
The second term on the right-hand side of \eqref{EqSi} is equal to
\[
8\|S_{\le d}^2\Delta P_{\le d}\hat{P}_{ > d}\|_2^2=8\| S_{\le d}^2\Delta S_{\le d}R_{\le d}\hat{P}_{>d}\|_2^2
\]
with $R_{\le d}=R_{\le d}(\lambda_{d+1})$ from \eqref{EqWeightOperR} and thus bounded by
\[
8\|S_{\le d}^2\Delta S_{\le d}\|_\infty^2 \|R_{\le d}\hat{P}_{ > d}\|_2^2.
\]
Similarly, the third term is bounded via
\begin{align*}
&8\sum_{j\le d}\sum_{l>d}\frac{1}{(\lambda_j-\lambda_{d+1})^2}\Bigg\| \sum_{k>d}\frac{P_j\Delta P_k\Delta \hat{P}_{l}}{\lambda_j-\lambda_k}\Bigg\|_2^2\\
&=8\sum_{j\le d}\frac{1}{(\lambda_j-\lambda_{d+1})^2}\Bigg\| \sum_{k>d}\frac{P_j\Delta P_k\Delta \hat{P}_{>d}}{\lambda_j-\lambda_k}\Bigg\|_2^2\\
&\le 16\sum_{j\le d}\frac{1}{(\lambda_j-\lambda_{d+1})^2}\Bigg\| \sum_{k>d}\frac{P_j\Delta P_k\Delta P_{>d}}{\lambda_j-\lambda_k}\Bigg\|_2^2\\
&+16\sum_{j\le d}\frac{1}{(\lambda_j-\lambda_{d+1})^2}\Bigg\|\sum_{k>d}\frac{P_j\Delta P_k\Delta S_{\le d}}{\lambda_j-\lambda_k}\Bigg\|_2^2 \|R_{\le d}\hat{P}_{>d}\|_2^2.
\end{align*}
Inserting these bounds into \eqref{EqSi} and using $\|R_{\le d}\hat{P}_{ > d}\|_2^2={\mathcal E}^{PCA}_{\le d}(\lambda_{d+1})$ from \eqref{EqExcRiskRepr}, \eqref{EqLinExpSinTheta} follows. Inserting \eqref{EqLinExpSinTheta} into the inequality
\[
\|P_{\le d}\hat P_{>d}\|_2^2\le 2\|P_{\le d}\hat{P}_{>d}-L_{1}(\Delta)\|_2^2+2\|L_{1}(\Delta)\|_2^2,
\]
and using that $\|P_{> d}\hat P_{\le d}\|_\infty^2\le 1$, we get
\begin{equation}\label{EqSinThetaFS1}
\|P_{\le d}\hat P_{>d}\|_2^2\le 4\|L_{1}(\Delta)\|_2^2+2R_1.
\end{equation}
Now, \eqref{EqLinExpSinThetaC1} follows from \eqref{EqSinThetaFS1} and $\|\hat{P}_{\le d}-P_{\le d}\|_2^2=2\|P_{\le d}\hat P_{>d}\|_2^2$. Inserting \eqref{EqSinThetaFS1} in combination with $\|P_{> d}\hat P_{\le d}\|_\infty^2\le \|P_{> d}\hat P_{\le d}\|_2^2=\|P_{\le d}\hat P_{>d}\|_2^2$ into \eqref{EqLinExpSinTheta}, we get
\begin{equation}\label{EqSinThetaFS2}
\|P_{\le d}\hat{P}_{>d}-L_{1}(\Delta)\|_2^2\le 4\|L_{1}(\Delta)\|_2^4+ 2R_1\|L_{1}(\Delta)\|_2^2+R_1,
\end{equation}
and \eqref{EqLinExpSinThetaC2} follows from inserting \eqref{EqSinThetaFS2} into
\begin{align*}
&|\|\hat{P}_{\le d}-P_{\le d}\|_2^2-2\|L_{1}(\Delta)\|_2^2|\\
&=2|\|P_{\le d}\hat P_{>d}\|_2^2-\|L_{1}(\Delta)\|_2^2|\\
&\le 2\|P_{\le d}\hat P_{>d}-L_{1}(\Delta)\|_2^2+4\|L_{1}(\Delta)\|_2\|P_{\le d}\hat P_{>d}-L_{1}(\Delta)\|_2.
\end{align*}
This completes the proof.
\end{proof}
Similarly, we have the following  upper bound for ${\mathcal E}^{PCA}_{\le d}(\lambda_{d+1} )$:

\begin{proposition}\label{LinExpER1} On the event $\{\hat{\lambda}_{d+1}-\lambda_{d+1}\le (\lambda_d-\lambda_{d+1})/2\}$ we have
\begin{align}
&\bigg\|R_{\le d}\hat P_{>d}-\sum_{j\le d}(\lambda_j-\lambda_{d+1})^{1/2}\sum_{k>d} \frac{P_j\Delta P_k}{\lambda_j-\lambda_k}\bigg\|_2^2\label{EqLinExpER1}\\
&\le \bigg\|\sum_{j\le d}(\lambda_j-\lambda_{d+1})^{1/2}\sum_{k>d} \frac{P_j\Delta P_k}{\lambda_j-\lambda_k}\bigg\|_2^2\|P_{> d}\hat P_{\le d}\|_\infty^2+R_2\nonumber
\end{align}
with $R_{\le d}=R_{\le d}(\lambda_{d+1})$ from \eqref{EqWeightOperR} and remainder term $R_2$ given by
\begin{align*}
R_2&=16\sum_{j\le d}\frac{1}{\lambda_j-\lambda_{d+1}}\Bigg\| \sum_{k>d}\frac{P_j\Delta P_k\Delta P_{>d}}{\lambda_j-\lambda_k}\Bigg\|_2^2+8\|S_{\le d}\Delta S_{\le d}\|_\infty^2 {\mathcal E}^{PCA}_{\le d}(\lambda_{d+1})\\
&+16\sum_{j\le d}\frac{1}{\lambda_j-\lambda_{d+1}}\Bigg\| \sum_{k>d}\frac{P_j\Delta P_k\Delta S_{\le d}}{\lambda_j-\lambda_k}\Bigg\|_2^2 {\mathcal E}^{PCA}_{\le d}(\lambda_{d+1}).
\end{align*}
In particular, letting
\[
L_{2}(\Delta)=\sum_{j\le d}(\lambda_j-\lambda_{d+1})^{1/2}\sum_{k>d} \frac{P_j\Delta P_k}{\lambda_j-\lambda_k},
\]
we have, on the event $\{\hat{\lambda}_{d+1}-\lambda_{d+1}\le (\lambda_d-\lambda_{d+1})/2\}$,
\begin{align}\label{EqLinExpER1C1}
{\mathcal E}^{PCA}_{\le d}(\lambda_{d+1} )&\le 4\|L_{2}(\Delta)\|_2^2+2R_2
\end{align}
and
\begin{align}
&|{\mathcal E}^{PCA}_{\le d}(\lambda_{d+1} )-\|L_{2}(\Delta)\|_2^2|\label{EqLinExpER1C2}\\
&\le C(\|L_{2}(\Delta)\|_2^2\|L_{1}(\Delta)\|_2^2+\|L_{2}(\Delta)\|_2^2\|L_{1}(\Delta)\|_2\nonumber\\
&+R_1\|L_{2}(\Delta)\|_2^2+\sqrt{R_1}\|L_{2}(\Delta)\|_2^2+\sqrt{R_2}\|L_{2}(\Delta)\|_2+R_2)\nonumber
\end{align}
with an absolute constant $C>0$ and $R_1$, $L_{1}(\Delta)$ from Proposition \ref{LinExpSinTheta}.
\end{proposition}
\begin{remark}
In Proposition \ref{LinExpER1} we did not establish a recursive argument. Instead, in order to bound the remainder terms, one can apply the excess risk bounds already derived in Propositions \ref{PropMainRes1} and \ref{MainRes2}.
\end{remark}
\begin{proof}
Inserting \eqref{eq:soev} into \eqref{EqLinExpER1}, we obtain
\begin{align*}
&\bigg\|R_{\le d}\hat P_{>d}-\sum_{j\le d}(\lambda_j-\lambda_{d+1})^{1/2}\sum_{k>d} \frac{P_j\Delta P_k}{\lambda_j-\lambda_k}\bigg\|_2^2\\
&\le \Bigg\|\sum_{j\le d}(\lambda_j-\lambda_{d+1})^{1/2} \sum_{k>d}\frac{P_j\Delta P_k\hat{P}_{\le d}}{\lambda_j-\lambda_k}\Bigg\|_2^2\\
&+2\Bigg\| \sum_{j\le d}(\lambda_j-\lambda_{d+1})^{1/2}\sum_{l>d}\frac{P_j\Delta P_{\le d}\hat{P}_{l}}{\lambda_j-\hat{\lambda}_l}\Bigg\|_2^2\\
&+2\Bigg\| \sum_{j\le d}(\lambda_j-\lambda_{d+1})^{1/2}\sum_{k>d}\sum_{l>d}\frac{P_j\Delta P_k\Delta \hat{P}_l}{(\lambda_j-\hat{\lambda}_l)(\lambda_j-\lambda_k)} \Bigg\|_2^2.
\end{align*}
Thus, on the event $\{\hat{\lambda}_{d+1}-\lambda_{d+1}\le (\lambda_d-\lambda_{d+1})/2\}$ we have
\begin{align}
&\bigg\|R_{\le d}\hat P_{>d}-\sum_{j\le d}(\lambda_j-\lambda_{d+1})^{1/2}\sum_{k>d} \frac{P_j\Delta P_k}{\lambda_j-\lambda_k}\bigg\|_2^2\label{EqQuadProofPr}\\
&\le \Bigg\|\sum_{j\le d}(\lambda_j-\lambda_{d+1})^{1/2} \sum_{k>d}\frac{P_j\Delta P_k}{\lambda_j-\lambda_k}\Bigg\|_2^2\|P_{>d}\hat{P}_{\le d}\|_\infty^2\nonumber\\
&+8\sum_{j\le d}\sum_{l>d}\frac{1}{\lambda_j-\lambda_{d+1}}\|P_j\Delta P_{\le d}\hat{P}_{l}\|_2^2\nonumber\\
&+8\sum_{j\le d}\sum_{l>d}\frac{1}{\lambda_j-\lambda_{d+1}}\Bigg\| \sum_{k>d}\frac{P_j\Delta P_k\Delta \hat{P}_l}{\lambda_j-\lambda_k} \Bigg\|_2^2.\nonumber
\end{align}
The second term on the right-hand side of \eqref{EqQuadProofPr} is equal to
\[
8\|S_{\le d}\Delta P_{\le d}\hat{P}_{>d}\|_2^2=8\|S_{\le d}\Delta S_{\le d}R_{\le d}\hat{P}_{>d}\|_2^2
\]
which can be bounded by
\begin{align*}
8\|S_{\le d}\Delta S_{\le d}\|_\infty^2\|R_{\le d}\hat{P}_{>d}\|_2^2=8\|S_{\le d}\Delta S_{\le d}\|_\infty^2{\mathcal E}^{PCA}_{\le d}(\lambda_{d+1}).
\end{align*}
Similarly, the third term on the right-hand side of \eqref{EqQuadProofPr} is bounded via
\begin{align*}
&8\sum_{j\le d}\sum_{l>d}\frac{1}{\lambda_j-\lambda_{d+1}}\Bigg\| \sum_{k>d}\frac{P_j\Delta P_k\Delta \hat{P}_{l}}{\lambda_j-\lambda_k}\Bigg\|_2^2\\
&=8\sum_{j\le d}\frac{1}{\lambda_j-\lambda_{d+1}}\Bigg\| \sum_{k>d}\frac{P_j\Delta P_k\Delta \hat{P}_{>d}}{\lambda_j-\lambda_k}\Bigg\|_2^2\\
&\le 16\sum_{j\le d}\frac{1}{\lambda_j-\lambda_{d+1}}\Bigg\| \sum_{k>d}\frac{P_j\Delta P_k\Delta P_{>d}}{\lambda_j-\lambda_k}\Bigg\|_2^2\\
&+16\sum_{j\le d}\frac{1}{\lambda_j-\lambda_{d+1}}\Bigg\| \sum_{k>d}\frac{P_j\Delta P_k\Delta S_{\le d}}{\lambda_j-\lambda_k}\Bigg\|_2^2{\mathcal E}^{PCA}_{\le d}(\lambda_{d+1})
\end{align*}
Inserting these bounds into \eqref{EqQuadProofPr}, \eqref{EqLinExpER1} follows. Inequality \eqref{EqLinExpER1C1} follows from inserting \eqref{EqLinExpER1} into the inequality
\begin{align*}
{\mathcal E}^{PCA}_{\le d}(\lambda_{d+1})=\|R_{\le d}\hat P_{>d}\|_2^2\le 2\|R_{\le d}\hat P_{>d}-L_{2}(\Delta)\|_2^2+2\|L_{2}(\Delta)\|_2^2
\end{align*}
and using $\|P_{> d}\hat P_{\le d}\|_\infty^2\le 1$. Moreover, we have
\begin{align*}
&|{\mathcal E}^{PCA}_{\le d}(\lambda_{d+1} )-\|L_{2}(\Delta)\|_2^2|\\
&\le \|R_{\le d}\hat P_{>d}-L_{2}(\Delta)\|_2^2+2\|L_{2}(\Delta)\|_2\|R_{\le d}\hat P_{>d}-L_{2}(\Delta)\|_2,
\end{align*}
and \eqref{EqLinExpER1C2} follows from inserting \eqref{EqLinExpER1} and $\|P_{> d}\hat P_{\le d}\|_\infty^2\le 4\|L_{1}(\Delta)\|_2^2+2R_1$ from \eqref{EqSinThetaFS1}.
\end{proof}

Following the same line of arguments, we have the following result for ${\mathcal E}^{PCA}_{> d}(\lambda_{d+1})$:

\begin{proposition}\label{LinExpER2} On the event $\{\hat{\lambda}_{d}-\lambda_{d}\ge -(\lambda_d-\lambda_{d+1})/2\}$ we have
\begin{align*}
&\bigg\|R_{>d}\hat P_{\le d}-\sum_{k> d}(\lambda_d-\lambda_k)^{1/2}\sum_{j\le d} \frac{P_k\Delta P_j}{\lambda_j-\lambda_k}\bigg\|_2^2\\
&\le \bigg\|\sum_{k> d}(\lambda_d-\lambda_k)^{1/2}\sum_{j\le d} \frac{P_k\Delta P_j}{\lambda_j-\lambda_k}\bigg\|_2^2\|P_{\le d}\hat P_{> d}\|_\infty^2+R_3
\end{align*}
with $R_{>d}=\sum_{k> d}(\lambda_d-\lambda_k)^{1/2}P_k$ and remainder term $R_3$ given by
\begin{align*}
R_3&=16\sum_{k> d}\frac{1}{\lambda_d-\lambda_{k}}\Bigg\| \sum_{j\le d}\frac{P_k\Delta P_j\Delta P_{\le d}}{\lambda_j-\lambda_k}\Bigg\|_2^2+8\|S_{> d}\Delta S_{> d}\|_\infty^2 {\mathcal E}^{PCA}_{> d}(\lambda_d)\\
&+16\sum_{k> d}\frac{1}{\lambda_d-\lambda_{k}}\Bigg\| \sum_{j\le d}\frac{P_k\Delta P_j\Delta S_{> d}}{\lambda_j-\lambda_k}\Bigg\|_2^2 {\mathcal E}^{PCA}_{> d}(\lambda_d)
\end{align*}
with $S_{>d}=\sum_{k> d} (\lambda_d-\lambda_k)^{-1/2}P_k$. In particular, on the event $\{\hat{\lambda}_{d}-\lambda_{d}\ge -(\lambda_d-\lambda_{d+1})/2\}$, we have
\begin{align*}
{\mathcal E}^{PCA}_{> d}(\lambda_d )&\le 4\bigg\|\sum_{k> d}(\lambda_d-\lambda_k)^{1/2}\sum_{j\le d} \frac{P_k\Delta P_j}{\lambda_j-\lambda_k}\bigg\|_2^2+2R_3.
\end{align*}
\end{proposition}

\subsection{Consequences for the $L^{k}$-norm}
In this section, we assume that Assumption \ref{SubGauss} holds. Then, similarly as in Section \ref{SecNewBounds}, we can take expectation (resp. the $L^{k}$-norm, $k\ge 1$) in Propositions \ref{LinExpSinTheta}, \ref{LinExpER1}, and \ref{LinExpER2}. Since this leads to lengthy remainder terms, we give the details in the case of the Hilbert-Schmidt distance. Analogous results hold for the excess risk. Let $k\ge 1$ be a natural number. First, by the triangle inequality, the Rosenthal inequality, and  Assumption \ref{SubGauss}, we have
\begin{align}
&\bigg(\mathbb{E}\bigg[\bigg\|\sum_{j\le d}\sum_{k>d} \frac{P_j\Delta P_k}{\lambda_j-\lambda_k}\bigg\|_2^{2k}\bigg]\bigg)^{1/k}\label{EqLkNormLinExp1}\\
&\le \sum_{j\le d}\sum_{k>d} \frac{(\mathbb{E}\|P_j\Delta P_k\|^{2k})^{1/k}}{(\lambda_j-\lambda_k)^{2}}\le C\sum_{j\le d}\sum_{k>d}\frac{\lambda_j\lambda_k}{n(\lambda_j-\lambda_k)^2}\nonumber
\end{align}
with a constant $C>0$ depending only on $k$ and $C_1$ (cf. \eqref{EqMomentConst} for the case $k=1$ and \eqref{eq:eq2} for the case $k=2$).
 Similarly, the first term in $R_1$ is bounded as follows
\begin{align}
&\Bigg(\mathbb{E}\Bigg[\Bigg\| \sum_{j\le d}\sum_{k>d}\frac{P_j\Delta P_k\Delta P_{>d}}{(\lambda_j-\lambda_{d+1})(\lambda_j-\lambda_k)}\Bigg\|_2^{2k}\Bigg]\Bigg)^{1/k}\label{EqLkNormLinExp2}\\
&\le C\sum_{j\le d}
\frac{\lambda_j\tr_{>d}(\Sigma)}{n^2(\lambda_j-\lambda_{d+1})^2}
\bigg(\sum_{k>d}\frac{\lambda_k}{\lambda_j-\lambda_k} \bigg)^2\nonumber
\end{align}
with a constant $C>0$ depending only on $k$ and $C_1$. The middle in $R_1$ can be bounded by using the Cauchy-Schwarz inequality in combination with the fact that bounds of the same order as presented in Proposition \ref{MainRes2} can be derived  for the $L^{2k}$-norm of ${\mathcal E}^{PCA}_{\le d}(\lambda_{d+1})$ (see also the comment after Proposition \ref{PropMainRes1}) and the inequality (see e.g. \cite[Corollary 2]{KL14})
\begin{align*}
&\left( \mathbb{E}\left[ \|S_{\le d}\Delta S_{\le d}\|_\infty^{2k}\right]\right) ^{1/k}\\
&\le C\bigg(\frac{\lambda_d}{\lambda_d-\lambda_{d+1}}\sum_{j\le d}\frac{\lambda_j}{n(\lambda_j-\lambda_{d+1})}\bigvee \bigg(\sum_{j\le d}\frac{\lambda_j}{n(\lambda_j-\lambda_{d+1})}\bigg)^2\bigg)
\end{align*}
with a constant $C>0$ depending only on $k$ and $C_1$.
Hence, if
\[
\frac{\lambda_d}{\lambda_d-\lambda_{d+1}}\sum_{j\le d}\frac{\lambda_j}{\lambda_j-\lambda_{d+1}}\le n/(16C_3^2),
\]
then we have
\begin{align}
&\left( \mathbb{E}\left[ \|S_{\le d}^2\Delta S_{\le d}\|_\infty^{2k} {\mathcal E}^{PCA}_{\le d}(\lambda_{d+1})^k\right]\right)^{1/k}\label{EqLkNormLinExp3}\\
&\le C\bigg(\frac{\lambda_d}{n^2(\lambda_d-\lambda_{d+1})^2}\bigg( \sum_{j\le d}\frac{\lambda_j}{\lambda_j-\lambda_{d+1}}\bigg)^2\bigg)\nonumber\\
&\ \ \ \ \cdot \bigg(\tr_{> d}(\Sigma)+\operatorname{tr}(\Sigma)\exp\bigg(-\frac{n(\lambda_{d}-\lambda_{d+1})^2}{16C_3^2\lambda_{d}^2}\bigg)\bigg)\nonumber
\end{align}
with a constant $C>0$ depending only on $k$ and $C_1$. Similarly, the third term in $R_1$ is bounded as follows
\begin{align}
&\Bigg(\mathbb{E}\Bigg[\Bigg\| \sum_{j\le d}\sum_{k>d}\frac{P_j\Delta P_k\Delta S_{\le d}}{(\lambda_j-\lambda_{d+1})(\lambda_j-\lambda_k)}\Bigg\|_2^{2k} {\mathcal E}^{PCA}_{\le d}(\lambda_{d+1})^k\Bigg]\Bigg)^{1/k}\label{EqLkNormLinExp4}\\
&\le C\bigg(\sum_{j\le d}\frac{\lambda_j}{n^3(\lambda_j-\lambda_{d+1})^2}
\bigg(\sum_{l>d}\frac{\lambda_l}{\lambda_j-\lambda_l} \bigg)^2\bigg( \sum_{m\le d}\frac{\lambda_m}{\lambda_m-\lambda_{d+1}} \bigg)^2\bigg)\nonumber\\
&\ \ \ \ \cdot\bigg(\tr_{> d}(\Sigma)+\operatorname{tr}(\Sigma)\exp\bigg(-\frac{n(\lambda_{d}-\lambda_{d+1})^2}{16C_3^2\lambda_{d}^2}\bigg)\bigg)\nonumber
\end{align}
with a constant $C>0$ depending only on $k$ and $C_1$. Finally, by Proposition \ref{RightDev} and the inequality $\|P_{\le d}-\hat P_{\le d}\|_2^2\le 2d$, if
\[
\frac{\lambda_{d+1}}{\lambda_d-\lambda_{d+1}}\sum_{k>d}\frac{\lambda_{k}}{\lambda_d-\lambda_{k}}\le 4n/C_3^2,
\]
then
\begin{align}\label{EqLkNormLinExp5}
&\left(\mathbb{E}\left[\mathbbm{1}\big(\hat{\lambda}_{d+1}-\lambda_{d+1}> (\lambda_d-\lambda_{d+1})/2\big)\|P_{\le d}-\hat P_{\le d}\|_2^{2k}\right]\right)^{1/k}\\
&\le 2d\exp\bigg(-\frac{n(\lambda_{d}-\lambda_{d+1})^2}{k(2C_3\lambda_{d})^2}\bigg).\nonumber
\end{align}
By \eqref{EqLkNormLinExp1}-\eqref{EqLkNormLinExp5} for $k=1$, we immediately get an upper bound for $\mathbb{E}[\|P_{\le d}-\hat P_{\le d}\|_2^2]$ as well as for $|\mathbb{E}[\|P_{\le d}-\hat P_{\le d}\|_2^2]-\mathbb{E}\|L_{1}(\Delta)\|_2^2|$.
\subsection{Applications}\label{AppAppli}
Let us illustrate our different bounds for exponential decay and polynomial decay.
\subsubsection*{Exponential decay}
Assume that for some $\alpha>0$, $\lambda_j=e^{-\alpha j}$ for every $j\ge 1$. Let us begin with applying Proposition \ref{LinExpSinTheta} to the Hilbert-Schmidt distance (note that under exponential decay, the eigenvalue expressions in \eqref{EqLkNormLinExp1}-\eqref{EqLkNormLinExp5} are easy to compute by using the inequality $\lambda_j-\lambda_k\ge (1-e^{-\alpha})\lambda_j$, valid for every $j<k$). First, by \eqref{EqLinExpSinThetaC1}, we get the upper bound
\begin{equation*}
\mathbb{E}[ \Vert \hat{P}_{\le d}-P_{\le d}\Vert_2^2]\le Cn^{-1}+Cd^2n^{-2},
\end{equation*}
provided that $d\le cn$, where $c,C>0$ are constants depending
only on $C_1$ and $\alpha$.
Moreover, assuming additionally  that $X$ is Gaussian, we have $\mathbb{E}\|L_{1}(\Delta)\|_2^2\ge (e^{\alpha}-1)^{-1}n^{-1}$ and by \eqref{EqLinExpSinThetaC2} we get the lower bound (after some computation)
\begin{equation*}
\mathbb{E}[ \Vert \hat{P}_{\le d}-P_{\le d}\Vert_2^2]\ge 2^{-1}(e^{\alpha}-1)^{-1}n^{-1},
\end{equation*}
provided that $d\le c\sqrt{n}$, where $c>0$ is a constants depending
only on $C_1$ and $\alpha$. The upper bound can be compared to the bounds by Mas and Ruymgaart \cite{MR15} and Koltchinskii and Lounici \cite{KL15a}, who treat spectral projectors via resolvents and holomorphic functional calculus. \cite[Theorem 5]{MR15} says that
\[
\mathbb{E}[ \Vert (\hat{P}_{\le d}-P_{\le d})u\Vert^2]\le Cd^2\log^2 (n)n^{-1}
\]
for all $d\ge 2$, $n\ge 2$, and certain unit vectors $u$. Since the left-hand side is bounded by $4$, the bound is only useful if $d\leq c\sqrt{n}$. In contrast, we establish sharper results in the larger range $d\leq cn$. Moreover, applying \cite[Lemma~2]{KL15a}, we have
\[
\mathbb{E}[ \|\hat{P}_{\le d}-P_{\le d}\|_2^2]\le  Cn^{-1}+Cde^{5\alpha d}n^{-2}.
\]
For the remainder term to be small, this requires $n^2$ to be much larger than $e^{5\alpha d}$, which our analysis avoids.

We now apply Proposition  \ref{LinExpER1} to the excess risk. In Section \ref{ExExpDecay} we showed that ${\mathcal E}^{PCA}_{\le d}(\lambda_{d+1})\le {\mathcal E}^{PCA}_d \le (1-e^{-\alpha})^{-1}{\mathcal E}^{PCA}_{\le d}(\lambda_{d+1})$. Thus, if $X$ is Gaussian, then \eqref{EqLinExpER1C2} in combination with Proposition \ref{MainRes2} and the analogues of \eqref{EqLkNormLinExp1}-\eqref{EqLkNormLinExp5} for $k=1$ gives
\begin{equation}\label{EqEROptBoundED}
C^{-1} d e^{-\alpha d}n^{-1}\le \mathbb{E}[{\mathcal E}^{PCA}_d]
\le Cd e^{-\alpha d}n^{-1},
\end{equation}
provided that $d\le cn$, where $c,C>0$ are constants depending
only on $\alpha$.

\subsubsection*{Polynomial decay}
Assume that  for some $\alpha>1$, $\lambda_j=j^{-\alpha}$ for every $j\ge 1$. Let us begin with applying Proposition \ref{LinExpSinTheta} to the Hilbert-Schmidt distance (note that under polynomial decay, the eigenvalue expressions in \eqref{EqLkNormLinExp1}-\eqref{EqLkNormLinExp5} can be easily computed using \eqref{EqEVExp1}-\eqref{EqEVExp4}). First, by \eqref{EqLinExpSinThetaC1}, we get
\begin{equation*}
\mathbb{E}[ \Vert \hat{P}_{\le d}-P_{\le d}\Vert_2^2]\le C(d^2\log (ed)n^{-1}+d^5\log^2 (ed)n^{-2}+d^7\log^4 (ed)n^{-3}),
\end{equation*}
provided that $d^2\log (ed)\le cn$, where $c,C>0$  are constants depending
only on $C_1$ and $\alpha$. Compared to \cite[Theorem 5]{MR15}, where the order $d^2\log^2 (n)\log^2 (d)n^{-1}$ is derived, this upper bound improves upon the $\log^2 (n)$-factor, but involves a remainder which might be harmful for $d^3n^{-1}$ large, but $d^2n^{-1}$ small. For this case higher order than linear expansions can extend the domain where the bound $d^2\log (ed)n^{-1}$ holds.

Assuming additionally  that $X$ is Gaussian, we have $\mathbb{E}\|L_{1}(\Delta)\|_2^2\ge c_1d^2\log (ed)n^{-1}$ (with $c_1>0$ depending only on $\alpha$) and thus we get the lower bound
\begin{equation*}
\mathbb{E}[ \Vert \hat{P}_{\le d}-P_{\le d}\Vert_2^2]\ge 2^{-1}c_1d^2\log (ed)n^{-1},
\end{equation*}
provided that $d^3\log (ed)\le cn$, where $c>0$  is a constant depending
only on $C_1$ and $\alpha$. 

We now turn to the excess risk. First, \eqref{EqLinExpER1C1} in combination with Proposition~\ref{MainRes2} and the analogues of \eqref{EqLkNormLinExp1}-\eqref{EqLkNormLinExp5} with $k=1$ gives
\begin{equation}\label{EqERUpperBoundPD}
\mathbb{E}[{\mathcal E}^{PCA}_{\le d}(\lambda_{d+1})] \le Cd^{2-\alpha}n^{-1},
\end{equation}
provided that $d^2\log^3 (d)\le cn$, where $c,C>0$  are constants depending
only on $C_1$ and $\alpha$. Moreover, assuming additionally that $X$ is Gaussian, we have $\mathbb{E}\|L_{2}(\Delta)\|_2^2\ge c_1 d^{2-\alpha}n^{-1}$  (with $c_1>0$ depending only on $\alpha$) and by \eqref{EqLinExpER1C2} we get the lower bound 
\begin{equation}\label{EqERLowerBoundPD}
\mathbb{E}[{\mathcal E}^{PCA}_{\le d}(\lambda_{d+1})]\ge 2^{-1}c_1 d^{2-\alpha}n^{-1},
\end{equation}
provided that $d^{5/2}\log (ed)\le cn$, where $c,C>0$  are constants depending
only on $C_1$ and $\alpha$. Finally, combining Propositions \ref{LinExpER1}, \ref{LinExpER2} with Theorem \ref{ThmGlobalLocalBound}, we obtain the
bound 
\begin{equation}\label{EqERLinExpPD}
\mathbb{E}[{\mathcal E}^{PCA}_{d}]\le C(d^{2-\alpha}n^{-1}+d^3\log^3 (ed)^3n^{-2}+d^5\log^5 (ed)n^{-3}),
\end{equation}
provided that $d^2\log (d)\le cn$, where $c,C>0$  are constants depending
only on $C_1$ and $\alpha$.

\subsection{Some eigenvalue expressions}
Suppose that for some $\alpha>1$, $\lambda_j=j^{-\alpha}$ for every $j\ge 1$. Then there is a constant $C>0$ depending only on $\alpha$ such that, for every $d\ge 1$,
\begin{align}
&\sum_{j\ne d}\frac{\lambda_j}{|\lambda_j-\lambda_d|}\le Cd\log (ed),\label{EqEVExp1}\\
&\sum_{j\ne d}\frac{\lambda_j}{(\lambda_j-\lambda_d)^2}\le Cd^{2+\alpha},\label{EqEVExp2}\\
&\sum_{j\le d}\sum_{k>d}\frac{\lambda_j\lambda_k}{\lambda_j-\lambda_k}\le Cd^{2-\alpha},\label{EqEVExp3}\\
&\sum_{j\le d}\sum_{k>d}\frac{\lambda_j\lambda_k}{(\lambda_j-\lambda_k)^2}\le Cd^2\log (ed).\label{EqEVExp4}
\end{align}
Moreover, these bounds are sharp in the sense that the reverse inequalities also hold, again with a constant depending only on $\alpha$. For \eqref{EqEVExp1} and \eqref{EqEVExp2}, see \cite[Lemma 7.13]{M15} and the references therein. Inequalities \eqref{EqEVExp3} and \eqref{EqEVExp4} can be shown similarly. Indeed, in order to prove \eqref{EqEVExp3}, decompose the sum as follows:
\begin{align}
\sum_{j\le d}\sum_{k>d}\frac{\lambda_j\lambda_k}{\lambda_j-\lambda_k}
&=\Big(\sum_{j\le d/2}\sum_{k>d}+\sum_{j\le d}\sum_{k>2d}+\sum_{d/2<j\le d}\sum_{d<k\le 2d}\Big)\frac{\lambda_j\lambda_k}{\lambda_j-\lambda_k}\label{EqEVExpDec}\\
&=:(I)+(II)+(III).\nonumber
\end{align}
Since $\alpha>1$, we have $j\lambda_j=j^{1-\alpha}>k^{1-\alpha}=k\lambda_k$ for every $j<k$ and thus
\begin{equation}\label{EqEVConvCond}
\frac{\lambda_j}{\lambda_j-\lambda_k}< \frac{k}{k-j}\qquad \text{for every }j<k.
\end{equation}
For $j\le d/2,k>d$ and $j\le d,k> 2d$, we have $k/(k-j)\le 2$. Combining this with \eqref{EqEVConvCond}, we get
\begin{equation}\label{EqSumPartI}
(I)\le \sum_{j\le d/2}\sum_{k>d} 2\lambda_k\le Cd^{2-\alpha},\quad (II)\le \sum_{j\le d}\sum_{k\ge 2d} 2\lambda_k\le Cd^{2-\alpha}.
\end{equation}
Moreover, by \eqref{EqEVConvCond}, we have
\begin{equation*}
(III)\le \sum_{d/2<j\le d}\sum_{d<k\le 2d}\frac{k\lambda_k}{k-j}\le d^{1-\alpha}\sum_{d/2<j\le d}\sum_{d<k\le 2d}\frac{1}{k-j}.
\end{equation*}
Now, for $l\ge 1$ the number of indices $(j,k)$ in the latter sum satisfying $k-j=l$ is less than or equal to $l$ if $l\le 2d$ and equal to $0$ otherwise. Hence,
\begin{equation}\label{EqSumPartIII}
(III)\le 2d^{2-\alpha}.
\end{equation}
Inserting \eqref{EqSumPartI} and \eqref{EqSumPartIII} into \eqref{EqEVExpDec}, \eqref{EqEVExp3} follows. For the reverse inequality, consider only the term $(III)$. By convexity, we have $(k^\alpha-j^{\alpha})/(k-j)\le \alpha k^{\alpha-1}$ for every $j<k$ and thus
\[
(III)=\sum_{d/2<j\le d}\sum_{d<k\le 2d}\frac{1}{k^\alpha-j^{\alpha}}\ge \alpha^{-1}(2d)^{1-\alpha}\sum_{d/2<j\le d}\sum_{d<k\le 2d}\frac{1}{k-j}.
\]
By the above argument, the last sum is lower bounded by a constant times $d$ and the claim follows. \eqref{EqEVExp4} follows from the same line of arguments, using also the inequality $\log(m+1)\le \sum_{l\le m}1/l\le \log(em)$, valid for every natural number $m\ge 1$.

\bibliographystyle{plain}
\bibliography{lit}

\end{document}